%% file: torelli-homology-torsion.tex
\documentclass[a4paper,12pt, reqno]{amsart}
\usepackage[a4paper, left=0.85in, right=0.85in, top=0.85in, bottom=0.85in]{geometry}

\usepackage{mathtools}
\usepackage{amssymb, amsfonts}
\usepackage{thmtools}
\usepackage{bm}
\usepackage{bbm}

\DeclareSymbolFontAlphabet{\mathbb}{AMSb}

\newcommand\N{\ensuremath{\mathbb{N}}}

\newcommand\Z{\ensuremath{\mathbb{Z}}}

\newcommand\Q{\ensuremath{\mathbb{Q}}}

\DeclareMathOperator{\Homeo}{Homeo}
\DeclareMathOperator{\Homeop}{\ensuremath{Homeo^{\,+}}}

\DeclareMathOperator{\Mod}{Mod}

\DeclareMathOperator{\rank}{rank}

\DeclareMathOperator{\I}{\mathcal{I}}

\DeclareMathOperator{\A}{\mathcal{A}}
\DeclareMathOperator{\B}{\mathbb{B}}

\DeclareMathOperator{\Sfr}{\mathfrak{S}}

\DeclareMathOperator{\Sp}{Sp}
\DeclareMathOperator{\SL}{SL}

\DeclareMathOperator{\Arf}{Arf}

\DeclareMathOperator{\sgn}{sgn}

\let\1\foobarbaz
\let\0\foobarbaz

\DeclareMathOperator{\ba}{\ensuremath{\bm{\mathrm{a}}}}
\DeclareMathOperator{\bb}{\ensuremath{\bm{\mathrm{b}}}}

\DeclareMathOperator{\be}{\ensuremath{\bm{\mathrm{e}}}}

\DeclareMathOperator{\bx}{\ensuremath{\bm{\mathrm{x}}}}
\DeclareMathOperator{\by}{\ensuremath{\bm{\mathrm{y}}}}

\DeclareMathOperator{\bU}{\ensuremath{\bm{\mathrm{U}}}}
\DeclareMathOperator{\bV}{\ensuremath{\bm{\mathrm{V}}}}
\DeclareMathOperator{\0}{\bm{0}}
\DeclareMathOperator{\1}{\bm{1}}

\let\Cap\foobarbar
\DeclareMathOperator{\HH}{\mathcal{H}}
\DeclareMathOperator{\SSS}{\mathcal{S}}
\DeclareMathOperator{\OP}{\mathcal{O}}
\DeclareMathOperator{\Cap}{\ensuremath{\mathcal{C}\!\!\!\!\:\,\,\mathit{ap}}}
\DeclareMathOperator{\CB}{\ensuremath{\mathcal{CB}}}
\DeclareMathOperator{\GP}{\ensuremath{\mathcal{G}}}
\DeclareMathOperator{\GCap}{\ensuremath{\mathcal{GC}\!\!\!\!\:\,\,\mathit{ap}}}

\usepackage{tikz}
\usetikzlibrary{
    shapes.geometric,
    calc,
    arrows.meta,
    decorations.markings,
    math,
}

\usepackage{pgfplots}
\pgfplotsset{compat=1.18}
\usepgfplotslibrary{polar}

\usetikzlibrary{cd}
\usetikzlibrary{babel}

\numberwithin{equation}{section}
\newtheorem{theorem}{Theorem}[section]
\newtheorem{proposition}[theorem]{Proposition}
\newtheorem{lemma}[theorem]{Lemma}
\newtheorem{corollary}[theorem]{Corollary}
\newtheorem{fact}[theorem]{Fact}

\theoremstyle{definition}

\newtheorem{remark}[theorem]{Remark}

\usepackage{mathdots}
\usepackage{csquotes}

\usepackage{enumitem}
\usepackage{graphicx}
\usepackage{caption}

\usepackage[hidelinks]{hyperref}

\usepackage{import}
\usepackage{float}
\usepackage{xifthen}
\usepackage{pdfpages}
\usepackage{transparent}

\usepackage[
    backend=biber,
    sorting=nyt,
    giveninits=true, 
    doi=false,
    url=false,
    isbn=false
]{biblatex}

\renewbibmacro*{volume+number+eid}{%
  \printfield{volume}%
  \setunit{\addcolon} 
  \printfield{number}%
  \setunit{\addcomma\space}%
  \printfield{eid}}

\renewbibmacro{in:}{}

\title{On torsion in the homology of the Torelli group}
\author{Andrei Vladimirov}
\address{Lomonosov Moscow State University, Russia}

\subjclass{57M07 (Primary); 20J05; 20J06 (Secondary)}
\keywords{mapping class group, Torelli group, abelian cycles, homology of
groups, Birman--Craggs--Johnson homomorphism}

\email{andreykaere1@gmail.com}
\date{}
    
\begin{document}
    \begin{abstract}
        Let \(S_g\) be a closed, oriented surface of genus \(g\), and let
        \(\Mod(S_g)\) denote its mapping class group. The \emph{Torelli
        group} \(\I_g\) is the subgroup of \(\Mod(S_g)\) consisting of mapping
        classes that act trivially on \(H_1(S_g)\). For any collection of
        pairwise disjoint, separating simple closed curves on \(S_g\), the
        corresponding Dehn twists pairwise commute and determine a homology
        class in \(H_k(\I_g)\), which is called an \emph{abelian cycle}. We
        prove that the subgroup of \(H_k(\I_g)\) generated by such abelian
        cycles is a \(\Z/2\Z\)-vector space for all \(k\), and that it is
        finite-dimensional for \(k = 2\) and \(g \geq 4\).
    \end{abstract}

    \maketitle
    \markleft{}
    \markright{}

    \input{./sections/introduction.tex}

    \input{./sections/basic-relations.tex}

    \input{./sections/case-k-2-g-gt-4.tex}

    \input{./sections/bp-relation.tex}

    \input{./sections/relations-in-hk.tex}

    \input{./sections/BCJ.tex}

    \input{./sections/h2-is-finite-dimensional.tex}

    \input{./sections/order-2-in-hk.tex}

    \printbibliography

\end{document}

%% file: sections/introduction.tex
\section{Introduction}

Let \(S_g^b\) be a compact oriented surface of genus \(g\) with \(b\) boundary
components. We omit \(b\) if it vanishes. The \emph{mapping class group} of
\(S_g^b\) is the group
\[
\Mod(S_g^b) = \pi_0(\Homeop(S_g^b, \partial S_g^b))
,\] 
where \(\Homeop(S_g^b, \partial S_g^b)\) is the group of orientation-preserving
homeomorphisms of \(S_g^b\) that fix each boundary component pointwise. There is a
natural action of \(\Mod(S_g)\) on \(H_1(S_g; \Z)\). This action preserves the
\emph{intersection form}, yielding a surjection 
\[
\Psi \colon \Mod(S_g) \to \Sp(2g, \Z)
.\] 
The \emph{Torelli group} of \(S_g\) is defined to be the kernel of \(\Psi\) and
is denoted by \(\I_g\). Similarly, one can define the Torelli group \(\I_g^1\)
for a surface with one boundary component. For \(b \in \{0,1\}\), we will denote by
\(\I_g^b = \I(S_g^b)\) the Torelli group of \(S_g^b\). 

It is a well-known fact that for \(g = 1\), the homomorphism \(\Psi\) is
actually an isomorphism, so \(\Mod(S_1) \cong \SL(2, \Z)\) and \(\I_1\) is
trivial. Mess showed~\cite{mess} that the group \(\I_2\) is an infinitely
generated free group. Johnson proved~\cite{johnson1} that the Torelli groups
\(\I_g\) and \(\I_g^1\) are finitely generated for \(g \geq 3\), and
explicitly determined a generating set.

It is an important and interesting problem to study the homology groups of the
Torelli groups \(\I_g\) for \(g \geq 3\). We denote the integral homology
\(H_k(G; \Z)\) by \(H_k(G)\).

Let us recall some known facts about homology groups \(H_k(\I_g)\):

\begin{itemize}
    \item Johnson~\cite{johnson3} computed the abelianization of the Torelli
    group \(\I_g\) for \(g \geq 3\):
    \[
    H_1(\I_g) \cong \Z^{\binom{2g}{3} - 2g} \oplus (\Z / 2 \Z)^{\binom{2g}{2} + 2g}
    .\] 
    
    \item Bestvina--Bux--Margalit~\cite{bestvina-bux-margalit} proved that for
    \(g \geq 3\), the  cohomological dimension of \(\I_g\) is \(3g - 5\),
    and that the top homology group \(H_{3g-5}(\I_g)\) is not finitely
    generated.

    \item Gaifullin~\cite{gaifullin} proved that \(H_k(\I_g)\) contains a free
    abelian group of infinite rank for \(2g-3 \leq k < 3g-5\), which
    implies that these homology groups are not finitely generated.
\end{itemize}

Studying \(H_2(\I_g)\) is especially important, as it is closely related to
the question of whether \(\I_g\) is finitely presented. Let us list some
known facts about \(H_2(\I_g)\):

\begin{itemize}
    \item In 2018, Kassabov and Putman~\cite{kassabov-putman} proved that for
    \(g \geq 3\), the group \(H_2(\I_g)\) is a finitely generated
    \(\Sp(2g, \Z)\)-module. However, this result provides no information about
    the finite generation of \(H_2(\I_g)\) as a group.

    \item In 2023, Minahan~\cite{minahan} proved that for \(g \geq 51\),
    the vector space \(H_2(\I_g; \Q)\) is finite-dimensional.

    \item In 2025, Minahan and Putman~\cite{minahan-putman} proved that for
    \(g \geq 5\), the vector space \(H_2(\I_g; \Q)\) is finite-dimensional.
    They further showed that \(H_2(\I_g; \Q)\) is an algebraic representation
    of \(\Sp(2g, \Z)\) for \(g \geq 6\). Consequently, using the computation
    of the algebraic part of \(H_2(\I_g; \Q)\) established by Kupers and
    Randal-Williams~\cite[Theorems 4.1 and 8.1]{kupers-williams} under the
    assumption that \(H_2(\I_g; \Q)\) is finite-dimensional for all large
    enough \(g\), Minahan and Putman were able to explicitly compute
    \(H_2(\I_g; \Q)\) for \(g \geq 6\).

\end{itemize}

The torsion in \(H_k(\I_g)\) remains largely mysterious for \(k \geq 2\).
In this paper, we study this torsion and, to state our main result precisely,
we first introduce some necessary definitions.

Let \(h_1, \dots, h_k\) be pairwise commuting elements of a group \(G\).
Consider the homomorphism \(\phi \colon \Z^k \to G\) that sends the generator of the
\(i^{\text{th}}\) factor to \(h_i\). The \emph{abelian cycle}
\(\A(h_1, \dots, h_k)\) is defined as the image of the standard generator
\(\mu_k \in H_k(\Z^k) \cong \Z\) under the pushforward map
\(\phi_{*} \colon H_k(\Z^k) \to H_k(G)\).

\subsection{Results}
We denote by \(H_k^{\mathrm{ab,sep}}(\I_g)\) the subgroup of \(H_k(\I_g)\)
generated by all abelian cycles of the form \(\A(T_{\delta_1}, \dots,
T_{\delta_k})\), where \(\delta_1, \dots, \delta_k\) are pairwise disjoint,
nonisotopic separating simple closed curves on \(S_g\), and \(T_{\delta}\) denotes the
left Dehn twist about \(\delta\). In this paper, we consider
this subgroup for \(k \geq 2\), since \(H_1^{\mathrm{ab,sep}}(\I_g)\)
was fully described by Johnson (see~\cite{johnson3}).

One of the main results of this paper is the
following theorem.

\begin{theorem} \label{thm:main-1}
    For \(g \geq 3\) and \(k \geq 2\), the following hold:
    \begin{itemize}
        \item If \(k < g\), the group \(H_k^{\mathrm{ab,sep}}(\I_g)\) is a
        \(\Z/2\Z\)-vector space, generated by abelian cycles
        \(\A(T_{\delta_1}, \dots, T_{\delta_k})\) with
        \(\delta_1, \dots, \delta_k\), pairwise disjoint, pairwise nonisotopic,
        separating simple closed curves of genus \(1\).

        \item If \(k \geq g\), the group \(H_k^{\mathrm{ab,sep}}(\I_g)\)
        is trivial.
    \end{itemize}
\end{theorem}

\begin{remark}
    If \(k > 2g - 3\), the group \(H_k^{\mathrm{ab,sep}}(\I_g)\) is trivial,
    since there are no abelian cycles of the form \(\A(T_{\delta_1}, \dots,
    T_{\delta_k})\). Indeed, a simple Euler characteristic argument shows
    that \(S_g\) admits at most \(2g - 3\) pairwise disjoint,
    pairwise nonisotopic separating simple closed curves.
\end{remark}

\begin{remark}
    Vautaw~\cite{vautaw} showed that any abelian subgroup of the Torelli group
    \(\I_g\) has rank at most \(2g - 3\). Since \(\I_g\) is torsion-free, the
    subgroup of \(H_k(\I_g)\) generated by abelian cycles is trivial for \(2g
    - 3 < k \leq 3g - 5\), and hence is not of interest.
\end{remark}

\begin{corollary}
    The subgroup \(H_k^{\mathrm{ab,sep}}(\I_g) \subset H_k(\I_g)\) is finitely
    generated if and only if there are only finitely many distinct abelian
    cycles of the form \(\A(T_{\delta_1}, \dots, T_{\delta_k})\). In
    particular, if there are infinitely many such abelian cycles, then
    \(H_k(\I_g)\) is not finitely generated.
\end{corollary}

The main result of this paper is the following theorem.

\begin{theorem} \label{thm:main-2}
    The \(\Z/2\Z\)-vector space \(H_2^{\mathrm{ab,sep}}(\I_g)\) is
    finite-dimensional for \(g \geq 4\).
\end{theorem}

\begin{remark}
    The question of whether \(H_2^{\mathrm{ab,sep}}(\I_3)\) is
    finite-dimensional remains open. In~\cite{gaifullin-spectral}, Gaifullin
    studied the spectral sequence associated with the action of \(\I_3\) on
    the complex of cycles introduced by
    Bestvina--Bux--Margalit~\cite{bestvina-bux-margalit}. His work provides
    evidence suggesting that \(H_2^{\mathrm{ab,sep}}(\I_3)\) may not be
    finitely generated.
\end{remark}

Brendle and Farb~\cite{brendle-farb} showed that
\(\A(T_{\delta_1}, T_{\delta_2})\) is nontrivial in \(H_2(\I_g)\) for two
disjoint, nonisotopic, separating simple closed curves
\(\delta_1\) and \(\delta_2\). However, the analogous statement does not
always hold in \(H_k(\I_g)\) for \(k > 2\). We now give a criterion for when
\(\A(T_{\delta_1}, \dots, T_{\delta_k})\) vanishes in \(H_k(\I_g)\).

\begin{theorem} \label{thm:main-3}
    Let \(\delta_1, \dots, \delta_k\) be pairwise disjoint, pairwise
    nonisotopic, separating simple closed curves on \(S_g\)
    with \(k < g\). Let \(\Sigma_1, \dots, \Sigma_{k+1}\) be the surfaces
    obtained by cutting \(S_g\) along \(\delta_1, \dots, \delta_k\). Then
    \(\A(T_{\delta_1}, \dots, T_{\delta_k}) = 0\) in \(H_k(\I_g)\) if and only
    if at least one of the surfaces \(\Sigma_1, \dots, \Sigma_{k+1}\) has
    genus \(0\).
\end{theorem}

\subsection{Basic properties of abelian cycles}
Throughout the paper, we will repeatedly use the following standard properties
of abelian cycles. These follow immediately from the identifications
\(H_k(\Z^k) \cong \wedge^k \Z^k \cong \Z\).

\begin{fact} \label{ft:lin}
    Let \(h_1', h_1, \dots, h_k \in G\). In \(H_k(G)\), we have
    \[
    \A(h_1 h_1', h_2, \dots, h_k) = \A(h_1, h_2, \dots, h_k) + \A(h_1', h_2, \dots, h_k)
    ,\]
    whenever all three abelian cycles are defined.
\end{fact}

\begin{fact} \label{ft:sym}
    Let \(h_1, \dots, h_k \in G\), and let \(\pi\) be a permutation of \(\{1,
    \dots, k\}\). Then in \(H_k(G)\), we have
    \[
    \A(h_{\pi(1)}, \dots, h_{\pi(k)}) = \sgn(\pi) \A(h_1, \dots, h_k)
    .\]
\end{fact}

\begin{corollary} \label{cor:comm}
    If \(x, y \in G\) commute with \(z_1, \dots, z_{n-1} \in G\), then
    \[
    \A([x, y], z_1, \dots, z_{n-1}) = 0 
    ,\]
    where \([x, y] = x y x^{-1} y^{-1}\) denotes the commutator of \(x\) and \(y\).
\end{corollary}

\subsection{Notation and conventions}
Let us introduce some notation and conventions used throughout the paper. A
closed curve is called \emph{simple} if it has no self-intersections, and a
simple closed curve is called \emph{essential} if it is not nullhomotopic.
Unless stated otherwise, by a \emph{curve} we mean an essential simple closed
curve. For a curve \(\alpha\) on \(S_g^b\), we denote by \([\alpha]\) its
homology class in \(H_1(S_g^b)\). For homology classes \(a, b \in
H_1(S_g^b)\), we denote by \(a \cdot b\) their algebraic intersection number.

For integral homology classes \(x, y, \dots\), we denote their reductions
modulo \(2\) by bold letters \(\bx, \by, \dots\). Similarly, for symplectic
summands \(V, U, \dots\) of \(H_1(S_g)\), their reductions
modulo \(2\) are denoted by \(\bV, \bU, \dots\).

We denote by \(g(S)\) the genus of a surface \(S\). The \emph{genus} of a
separating simple closed curve \(\gamma\) is defined as the smaller of the
genera of the two surfaces bounded by \(\gamma\) and is denoted by
\(g(\gamma)\). For an element \(h \in G\), we denote by \([h]\) its homology
class in \(H_1(G)\).

\subsection{Outline of the paper}
This paper is organized as follows. In \S\ref{sec:basic-relations}, we review
known relations in the mapping class group and the Torelli group. In
\S\ref{sec:case-k-2-g-gt-4}, we prove that abelian cycles of the form
\(\A(T_{\delta_1}, T_{\delta_2})\) have order~\(2\) in \(H_2(\I_g)\) for \(g
\geq 4\) (Proposition~\ref{prop:abelian-cycle-order}). In
\S\ref{sec:bp-relation}, we establish an auxiliary relation
(Proposition~\ref{prop:bp-relation}) and use it to prove the main
relations (Propositions~\ref{lem:sum-relation-1} and~\ref{lem:sum-relation-2})
in \S\ref{sec:relations-in-hk}. There, we also prove
Theorem~\ref{thm:main-1} for \(k = 2\) and \(g \geq 3\)
(Propositions~\ref{prop:ab-cyc-ord-gen-3} and~\ref{prop:ab-sep-gen}). In
\S\ref{sec:BCJ}, we derive a key relation between abelian cycles
(Proposition~\ref{prop:key-relation}) using the
Birman--Craggs--Johnson homomorphism. In \S\ref{sec:h2-fin-dim}, we prove
Theorem~\ref{thm:main-2}. Finally, in \S\ref{sec:order-2-hk}, we complete the
proof of Theorem~\ref{thm:main-1} by treating the remaining cases. There, we
also prove Theorem~\ref{thm:main-3}.

\subsection{Acknowledgments}
The author is grateful to his supervisor,
A.~A.~Gaifullin, for suggesting the topic of this work and for his guidance
and constructive feedback, which greatly improved the paper.
The author wishes to thank the anonymous referee for a very careful reading,
catching some typos and pointing out a flaw in the proof of
Theorem~\ref{thm:main-3} in an earlier version of this text.

The author was supported by the Theoretical Physics and Mathematics
Advancement Foundation \textquote{BASIS} (grant 25-8-2-20-1).

%% file: sections/basic-relations.tex
\section{Basic relations in \(\Mod(S)\) and \(\I(S)\)}
\label{sec:basic-relations}

In this section, we recall some known relations in the groups \(\Mod(S)\) and
\(\I(S)\). We begin by restating the following result of Johnson
(see~\cite[Corollary from Lemma 3]{johnson3}).

\begin{proposition} \label{prop:johnson3}
    Let \(\delta\) be a separating simple closed curve on \(S_g^b\) with \(g
    \geq 3\) and \(b \in \{0, 1\}\). Then there exist elements \(h_1,
    f_1, \dots, h_m, f_m \in \I_g^b\) such that
    \[
    T_{\delta}^2 = [h_1, f_1] \cdots [h_m, f_m]
    .\] 
    In particular, we have
    \begin{equation} \label{eq:johnson-ord-2}
    2[T_\delta] = 0 \quad \text{in } H_1(\I_g^b), \quad b \in \{0, 1\}
    .\end{equation}
\end{proposition}

\begin{corollary} \label{cor:johnson3}
    Let \(\delta\) be a simple closed curve on \(S_g\) bounding a subsurface
    \(\Sigma\) of genus at least \(3\). Then there exist elements
    \(h_1, f_1, \dots, h_m, f_m \in \I_g\) that are all supported on \(\Sigma\) such that
    \[
    T_\delta^2 = [h_1, f_1] \cdots [h_m, f_m]
    .\]
\end{corollary}

One of the most important relations in Torelli group theory is the
\emph{lantern relation} in \(\Mod(S_g)\) (see~\cite[Proposition 5.1]{farbmarg}).

\begin{proposition}[Lantern relation]
    Let \(b_1, b_2, b_3, b_4\) be the boundary components of \(S_0^4\), and let
    \(x, y, z\) be the simple closed curves on \(S_0^4\) as shown in
    Figure~\ref{fig:lantern-relation}. In \(\Mod(S_0^4)\), the following
    relation holds
    \[
    T_{b_1} T_{b_2} T_{b_3} T_{b_4} = T_x T_y T_z
    .\]
    
    \begin{figure}[H]
        \centering
    \def\svgwidth{.8\columnwidth}
    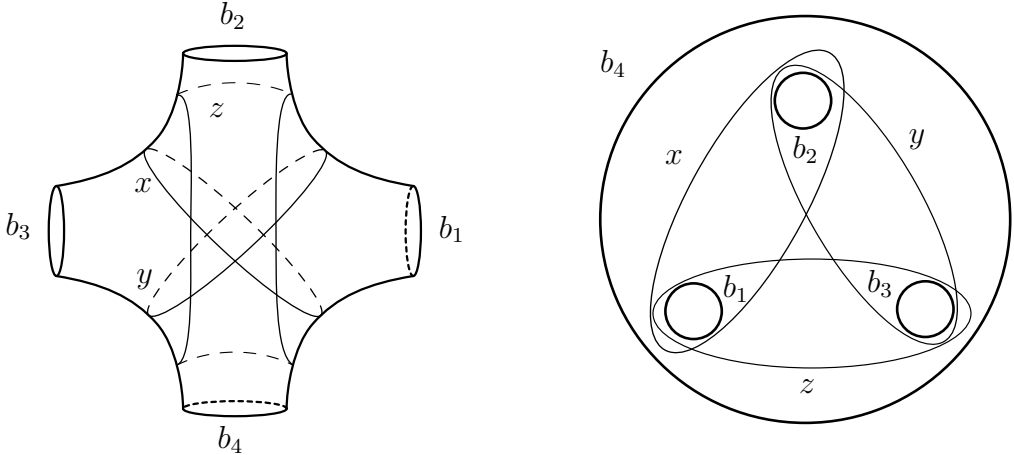

        \caption{Two views of the lantern relation.}
        \label{fig:lantern-relation}
    \end{figure}
    
    For any orientation-preserving embedding \(S_0^4 \hookrightarrow S\), the
    corresponding curves in \(S\) satisfy the same relation in \(\Mod(S)\).
\end{proposition}

%% file: figures/lantern-relation.pdf_tex
\begingroup%
  \makeatletter%
  \providecommand\color[2][]{%
    \errmessage{(Inkscape) Color is used for the text in Inkscape, but the package 'color.sty' is not loaded}%
    \renewcommand\color[2][]{}%
  }%
  \providecommand\transparent[1]{%
    \errmessage{(Inkscape) Transparency is used (non-zero) for the text in Inkscape, but the package 'transparent.sty' is not loaded}%
    \renewcommand\transparent[1]{}%
  }%
  \providecommand\rotatebox[2]{#2}%
  \newcommand*\fsize{\dimexpr\f@size pt\relax}%
  \newcommand*\lineheight[1]{\fontsize{\fsize}{#1\fsize}\selectfont}%
  \ifx\svgwidth\undefined%
    \setlength{\unitlength}{305.17099936bp}%
    \ifx\svgscale\undefined%
      \relax%
    \else%
      \setlength{\unitlength}{\unitlength * \real{\svgscale}}%
    \fi%
  \else%
    \setlength{\unitlength}{\svgwidth}%
  \fi%
  \global\let\svgwidth\undefined%
  \global\let\svgscale\undefined%
  \makeatother%
  \begin{picture}(1,0.43480358)%
    \lineheight{1}%
    \setlength\tabcolsep{0pt}%
    \put(0,0){\includegraphics[width=\unitlength,page=1]{lantern-relation.pdf}}%
    \put(0.14551707,0.16621376){\color[rgb]{0,0,0}\makebox(0,0)[t]{\lineheight{1.25}\smash{\begin{tabular}[t]{c}$y$\end{tabular}}}}%
    \put(0.14425604,0.25739681){\color[rgb]{0,0,0}\makebox(0,0)[t]{\lineheight{1.25}\smash{\begin{tabular}[t]{c}$x$\end{tabular}}}}%
    \put(0.44820022,0.21152288){\color[rgb]{0,0,0}\makebox(0,0)[t]{\lineheight{1.25}\smash{\begin{tabular}[t]{c}$b_1$\end{tabular}}}}%
    \put(0.23239417,0.42359682){\color[rgb]{0,0,0}\makebox(0,0)[t]{\lineheight{1.25}\smash{\begin{tabular}[t]{c}$b_2$\end{tabular}}}}%
    \put(0.02093909,0.21451787){\color[rgb]{0,0,0}\makebox(0,0)[t]{\lineheight{1.25}\smash{\begin{tabular}[t]{c}$b_3$\end{tabular}}}}%
    \put(0.22984514,0.00227088){\color[rgb]{0,0,0}\makebox(0,0)[t]{\lineheight{1.25}\smash{\begin{tabular}[t]{c}$b_4$\end{tabular}}}}%
    \put(0,0){\includegraphics[width=\unitlength,page=2]{lantern-relation.pdf}}%
    \put(0.21659893,0.33175328){\color[rgb]{0,0,0}\makebox(0,0)[t]{\lineheight{1.25}\smash{\begin{tabular}[t]{c}$z$\end{tabular}}}}%
    \put(0,0){\includegraphics[width=\unitlength,page=3]{lantern-relation.pdf}}%
    \put(0.60715081,0.37639157){\color[rgb]{0,0,0}\makebox(0,0)[t]{\lineheight{1.25}\smash{\begin{tabular}[t]{c}$b_4$\end{tabular}}}}%
    \put(0,0){\includegraphics[width=\unitlength,page=4]{lantern-relation.pdf}}%
    \put(0.72814697,0.15286577){\color[rgb]{0,0,0}\makebox(0,0)[t]{\lineheight{1.25}\smash{\begin{tabular}[t]{c}$b_1$\end{tabular}}}}%
    \put(0.7969551,0.29028092){\color[rgb]{0,0,0}\makebox(0,0)[t]{\lineheight{1.25}\smash{\begin{tabular}[t]{c}$b_2$\end{tabular}}}}%
    \put(0.8703926,0.15785653){\color[rgb]{0,0,0}\makebox(0,0)[t]{\lineheight{1.25}\smash{\begin{tabular}[t]{c}$b_3$\end{tabular}}}}%
    \put(0,0){\includegraphics[width=\unitlength,page=5]{lantern-relation.pdf}}%
    \put(0.66736179,0.28488792){\color[rgb]{0,0,0}\makebox(0,0)[t]{\lineheight{1.25}\smash{\begin{tabular}[t]{c}$x$\end{tabular}}}}%
    \put(0.9067952,0.30551596){\color[rgb]{0,0,0}\makebox(0,0)[t]{\lineheight{1.25}\smash{\begin{tabular}[t]{c}$y$\end{tabular}}}}%
    \put(0.79865918,0.05729448){\color[rgb]{0,0,0}\makebox(0,0)[t]{\lineheight{1.25}\smash{\begin{tabular}[t]{c}$z$\end{tabular}}}}%
  \end{picture}%
\endgroup%

%% file: sections/case-k-2-g-gt-4.tex
\section{Abelian cycles have order \(2\) for \(k = 2\) and \(g \geq 4\)}
\label{sec:case-k-2-g-gt-4}

In this section, we use Proposition~\ref{prop:johnson3} and
Corollary~\ref{cor:johnson3} to extend~\eqref{eq:johnson-ord-2} to
\(H_2(\I_g)\) for \(g \geq 4\). For the case \(g = 3\), we use
certain auxiliary relations in \(H_2(\I_g)\)
(see~\S\ref{sec:bp-relation} and~\S\ref{sec:relations-in-hk}).

\begin{proposition} \label{prop:abelian-cycle-order}
Let \(\delta_1, \delta_2\) be disjoint, nonisotopic, separating simple closed
curves on \(S_g\) with \(g \geq 4\). Then we have
\[
2 \A(T_{\delta_1}, T_{\delta_2}) = 0
.\]
\end{proposition}

\begin{proof}
    Since \(g \geq 4\) and the separating curves \(\delta_1, \delta_2\) are disjoint, 
    one of them bounds a subsurface \(\Sigma\) of genus at least \(3\). 
    For concreteness, we will assume that \(\delta_1\) bounds such a subsurface. 
    If \(\delta_2\) lies in \(\Sigma\), then by Proposition~\ref{prop:johnson3} we have
    \[
    T_{\delta_2}^2 = [h_1, f_1] \cdots [h_m, f_m]
    ,\] 
    where \(h_j\) and \(f_j\) are supported on \(\Sigma\) for all \(j = 1,
    \dots, m\). It follows from Facts~\ref{ft:lin},~\ref{ft:sym} and
    Corollary~\ref{cor:comm} that
    \[
    2 \A(T_{\delta_1}, T_{\delta_2}) = \A(T_{\delta_1}, T_{\delta_2}^2) =
    \A(T_{\delta_1}, [h_1, f_1]) + \cdots +
    \A(T_{\delta_1}, [h_m, f_m]) = 0
    ,\] 
    since each \(h_j\) and \(f_j\) commutes with \(T_{\delta_1}\).

    If \(\delta_2\) does not lie in \(\Sigma\), then by
    Corollary~\ref{cor:johnson3} we have
    \[
    T_{\delta_1}^2 = [\widetilde{h}_1, \widetilde{f}_1] \cdots
    [\widetilde{h}_n, \widetilde{f}_n]
    ,\] 
    where \(\widetilde{h}_j\) and \(\widetilde{f}_j\) are supported on
    \(\Sigma\) for all \(j = 1, \dots, n\).
    It follows again from Facts~\ref{ft:lin},~\ref{ft:sym} and
    Corollary~\ref{cor:comm} that
    \[
    2 \A(T_{\delta_1}, T_{\delta_2}) = \A(T_{\delta_1}^2, T_{\delta_2}) =
    \A([\widetilde{h}_1, \widetilde{f}_1], T_{\delta_2}) + \cdots +
    \A([\widetilde{h}_n, \widetilde{f}_n], T_{\delta_2}) = 0
    .\] 
\end{proof}

\begin{corollary} \label{cor:z2-vector-space}
    For \(g \geq 4\), the group \(H_2^{\mathrm{ab,sep}}(\I_g)\) is a
    \(\Z/2\Z\)-vector space.
\end{corollary}

%% file: sections/bp-relation.tex
\section{The BP-relation in \(H_k(\I_g)\)} \label{sec:bp-relation}

In this section, we derive an auxiliary relation in \(H_k(\I_g)\).
We begin by introducing the necessary terminology.

A pair of simple closed curves \((\alpha, \alpha')\) is called a \emph{bounding pair} if
\(\alpha\) and \(\alpha'\) are disjoint, nonisotopic, nonseparating simple
closed curves on \(S_g\) such that \(\alpha \cup \alpha'\) separates
\(S_g\).
We say that a bounding pair \((\alpha, \alpha')\) is \emph{oriented} if we fix
one of the two possible orientations. By definition, a bounding pair satisfies
\([\alpha] = \pm [\alpha']\), and choosing an orientation means simultaneously
orienting \(\alpha\) and \(\alpha'\) so that \([\alpha] = [\alpha']\).

Cutting \(S_g\) along \(\alpha \cup \alpha'\) yields two surfaces \(R_1\) and
\(R_2\), with \(R_1\) lying \emph{to the left} of \(\alpha\). Let \(\HH_1 =
H_1(R_1)\) and \(\HH_2 = H_1(R_2)\) denote the subspaces of \(H_1(S_g)\)
induced by the inclusions \(R_1 \hookrightarrow S_g\) and \(R_2
\hookrightarrow S_g\). We set \(\HH_{\alpha, \alpha'} = (\HH_1, \HH_2)\).

\begin{remark}
    For an oriented bounding pair \((\alpha, \alpha')\), it is \emph{not true}
    that \(\HH_{\alpha, \alpha'} = \HH_{\alpha', \alpha}\). In fact,
    \(\HH_{\alpha, \alpha'} = \HH_{\overline{\alpha}', \overline{\alpha}}\),
    where \(\overline{\alpha}\) denotes \(\alpha\) with the reversed orientation.
\end{remark}

\begin{proposition}[BP-relation] \label{prop:bp-relation}
    Let \((\alpha, \alpha')\) and \((\beta, \beta')\) be oriented bounding
    pairs such that \(\HH_{\alpha, \alpha'} = \HH_{\beta, \beta'}\). Let
    \(\delta_1, \dots, \delta_{k-1}\) be pairwise disjoint, pairwise nonisotopic
    separating simple closed curves on \(S_g\), each disjoint from
    \(\alpha, \alpha', \beta, \beta'\). Then in \(H_k(\I_g)\) we have
    \[
    \A(T_{\alpha} T_{\alpha'}^{-1}, T_{\delta_1}, \dots, T_{\delta_{k-1}})
    = \A(T_{\beta} T_{\beta'}^{-1}, T_{\delta_1}, \dots, T_{\delta_{k-1}})
    .\] 
\end{proposition}

To prove this proposition, we first state the following lemma, which
immediately implies the relation.

\begin{lemma} \label{lem:BP-splitting}
    Let \(\alpha, \alpha', \beta, \beta'\) be oriented, nonseparating simple
    closed curves on \(S_g\) representing the same homology class in \(H_1(S_g)\),
    such that \(\alpha\) is disjoint from \(\alpha'\) and \(\beta\) is disjoint from
    \(\beta'\).

    \begin{enumerate}
        \item If the oriented bounding pairs \((\alpha, \alpha')\) and
        \((\beta, \beta')\) satisfy \(\HH_{\alpha, \alpha'} = \HH_{\beta,
        \beta'}\), then there exists \(f \in \I_g\) such that \(f(\alpha) =
        \beta\) and \(f(\alpha') = \beta'\).

        \item Moreover, if \(\delta_1, \dots, \delta_n\) are pairwise
        disjoint, pairwise nonisotopic, separating simple closed curves disjoint from
        \(\alpha, \alpha', \beta, \beta'\), then \(f\) can be chosen so that 
        \(f(\delta_j) = \delta_j\) for all \(j = 1, \dots, n\).
    \end{enumerate}
\end{lemma}

\begin{proof}
    Let \(\Sigma_1, \dots, \Sigma_{n+1}\) be the subsurfaces of \(S_g\)
    obtained by cutting along \(\delta_1, \dots, \delta_n\) (where \(n = 0\)
    and \(\Sigma_1 = S_g\) if there are no \(\delta_i\)'s).  Since
    \(\HH_{\alpha, \alpha'} = \HH_{\beta, \beta'}\), all the four curves lie in
    some \(\Sigma \in \{\Sigma_1, \dots, \Sigma_{n+1}\}\).  Hence, it suffices
    to construct a homeomorphism \(f \in \Homeo(\Sigma, \partial \Sigma)\)
    such that:
    \begin{itemize}
        \item it takes \((\alpha, \alpha')\) to \((\beta, \beta')\); and
        \item its extension by the identity on \(S_g\) acts trivially on
        \(H_1(S_g)\).
    \end{itemize}

    Let \(N_1, N_2 \subset \Sigma\) be the subsurfaces obtained by cutting
    \(\Sigma\) along \(\alpha \cup \alpha'\), with \(N_1\) lying to the left
    of \(\alpha\). Similarly, let \(M_1, M_2 \subset \Sigma\) be the
    subsurfaces obtained by cutting \(\Sigma\) along \(\beta \cup \beta'\),
    with \(M_1\) lying to the left of \(\beta\). Since \(\HH_{\alpha, \alpha'}
    = \HH_{\beta, \beta'}\), it follows that \(N_1\) is homeomorphic to
    \(M_1\) and \(N_2\) is homeomorphic to \(M_2\).

    We can choose a homeomorphism \(\phi_1 \colon N_1 \to M_1\) fixing each
    \(\delta_j\) and taking \((\alpha, \alpha')\) to \((\beta, \beta')\).
    Moreover, there exists a homeomorphism \(\psi_1 \in \Homeo(M_1, \partial
    M_1)\) such that the extension of \(\psi_1 \circ \phi_1\) to \(S_g\) by
    the identity acts trivially on \(H_1(S_g)\).

    Similarly, we choose \(\phi_2 \colon N_2 \to M_2\)
    and \(\psi_2 \in \Homeo(M_2, \partial M_2)\) so that \(\psi_2 \circ
    \phi_2\) takes \((\alpha, \alpha')\) to \((\beta, \beta')\) and its
    extension to \(S_g\) by the identity acts trivially on \(H_1(S_g)\).

    Then \(\psi_1 \circ \phi_1 \cup \psi_2 \circ \phi_2 \in \Homeo(\Sigma,
    \partial \Sigma)\) is the desired homeomorphism, and the lemma follows.
\end{proof}

\begin{proof}[Proof of Proposition~\ref{prop:bp-relation}]
    By Lemma~\ref{lem:BP-splitting}, there exists \(f \in \I_g\) taking \((\alpha,
    \alpha')\) to \((\beta, \beta')\) and fixing each \(\delta_j\), \(j = 1, \dots, k-1\).
    Applying \(f_*\) to the abelian cycle, we obtain
    \[
    \A(T_{\alpha} T_{\alpha'}^{-1}, T_{\delta_1}, \dots,
    T_{\delta_{k-1}}) =
    f_{*}(\A(T_{\alpha} T_{\alpha'}^{-1}, T_{\delta_1}, \dots,
    T_{\delta_{k-1}})) = \A(T_{\beta} T_{\beta'}^{-1}, T_{\delta_1},
    \dots, T_{\delta_{k-1}})
    ,\]
    which proves the proposition.
\end{proof}

%% file: sections/relations-in-hk.tex
\section{Main relations in \(H_k^{\mathrm{ab,sep}}(\I_g)\)}
\label{sec:relations-in-hk}

In this section, we prove two auxiliary relations in \(H_k(\I_g)\). Recall
that a surface homeomorphic to a sphere with three boundary components is
called a \emph{pair of pants}.

\begin{lemma} \label{lem:sum-relation-1}
    Consider separating simple closed curves \(\gamma_1, \gamma_2, \gamma_3\)
    and \(\delta_1, \dots, \delta_{k-1}\) on \(S_g\) such that:
    \begin{enumerate}
        \item the curves \(\gamma_1, \gamma_2, \gamma_3\) decompose \(S_g\)
        into a pair of pants and three subsurfaces \(\Sigma_1, \Sigma_2,
        \Sigma_3\), where \(\partial \Sigma_j = \gamma_j\) for \(j = 1, 2,
        3\); 
        \item \(\Sigma_2\) has genus \(1\);
        \item the curves \(\delta_1, \dots, \delta_{k-1}\) are pairwise
        disjoint, pairwise nonisotopic, and lie in \(\Sigma_1, \Sigma_2,
        \Sigma_3\).
    \end{enumerate}
    Then in \(H_k(\I_g)\) we have
    \[
    \A(T_{\gamma_1}, T_{\delta_1}, \dots, T_{\delta_{k-1}}) + 
    \A(T_{\gamma_3}, T_{\delta_1}, \dots, T_{\delta_{k-1}}) = 
    \A(T_{\gamma_2}, T_{\delta_1}, \dots, T_{\delta_{k-1}})
    .\] 
\end{lemma}

\begin{remark}
    In Lemmas~\ref{lem:sum-relation-1} and~\ref{lem:sum-relation-2}, some of
    the curves \(\delta_1, \dots, \delta_{k-1}\) may be isotopic to the curves
    \(\gamma_1, \gamma_2, \gamma_3\).
\end{remark}

\begin{proof}
    We first consider the case in which none of the curves \(\delta_1, \dots,
    \delta_{k-1}\) is isotopic to \(\gamma_2\).
    
    Let \(\varepsilon, \varepsilon', \varepsilon''\) be curves on \(S_g\),
    arranged as shown in Figure~\ref{fig:sum-relation-1}. Applying the lantern
    relation, we obtain
    \begin{align} \label{eq:lantern}
    T_{\varepsilon}^2 T_{\gamma_1} T_{\gamma_3} = T_{\gamma_2}
    T_{\varepsilon'} T_{\varepsilon''}
    .\end{align}
    It follows from Fact~\ref{ft:lin} that
    \begin{equation} \label{eq:sum-rel}
    \begin{aligned}
        \A(T_{\gamma_1}, T_{\delta_1}, \dots, T_{\delta_{k-1}}) + 
        \A(T_{\gamma_3}, T_{\delta_1}, \dots, T_{\delta_{k-1}}) &= 
        \A(T_{\gamma_1} T_{\gamma_3}, T_{\delta_1}, \dots, T_{\delta_{k-1}})  \\
        &= \A(T_{\gamma_2} T_{\varepsilon'} T_{\varepsilon''} T_{\varepsilon}^{-2},
        T_{\delta_1}, \dots, T_{\delta_{k-1}}) \\
        &= \A(T_{\gamma_2}, T_{\delta_1}, \dots, T_{\delta_{k-1}}) \\
            & \qquad \qquad + 
            \A(T_{\varepsilon'} T_{\varepsilon}^{-1}, T_{\delta_1}, \dots,
            T_{\delta_{k-1}}) \\ 
            & \qquad \qquad -
            \A(T_{\varepsilon} T_{\varepsilon''}^{-1}, T_{\delta_1}, \dots,
            T_{\delta_{k-1}}).
    \end{aligned}
    \end{equation}

    \begin{figure}[H]
        \centering
    \def\svgwidth{.5\columnwidth}
    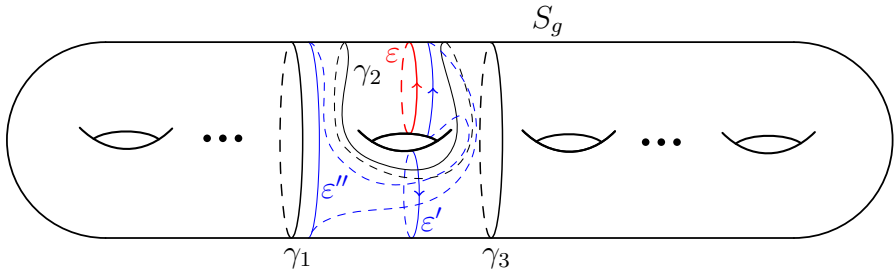

        \caption{The curves \(\varepsilon, \varepsilon', \varepsilon''\) on
        \(S_g\).}
        \label{fig:sum-relation-1}
    \end{figure}
    
    We have \(\HH_{\varepsilon, \varepsilon''} = \HH_{\varepsilon',
    \varepsilon} = (\HH_1, \HH_2)\), where 
    \[
    \HH_1 = H_1(\Sigma_3) \oplus \langle x \rangle, \qquad \HH_2 = H_1(\Sigma_1) \oplus \langle x \rangle
    ,\] 
    and \(x = [\varepsilon] = [\varepsilon'] = [\varepsilon'']\). It then
    follows from the BP-relation (Proposition~\ref{prop:bp-relation}) that
    \begin{align*} 
    \A(T_{\varepsilon'} T_{\varepsilon}^{-1}, T_{\delta_1}, \dots, T_{\delta_{k-1}}) =
    \A(T_{\varepsilon} T_{\varepsilon''}^{-1}, T_{\delta_1}, \dots, T_{\delta_{k-1}})
    .\end{align*}
    The lemma now follows from \eqref{eq:sum-rel}.

    Now we turn to the case in which some of \(\delta_1, \dots, \delta_{k-1}\)
    are isotopic to \(\gamma_2\). If at least two curves \(\delta_i\) and
    \(\delta_j\) are isotopic to \(\gamma_2\), then the statement of the lemma
    is trivial. Thus, without loss of generality, we assume that \(\delta_1\)
    is isotopic to \(\gamma_2\) and that \(\delta_2, \dots, \delta_{k-1} \in
    \Sigma_1 \cup \Sigma_3\). In this case, the desired relation takes the
    form:
    \[
    \A(T_{\gamma_1}, T_{\gamma_2}, T_{\delta_2}, \dots, T_{\delta_{k-1}}) + 
    \A(T_{\gamma_3}, T_{\gamma_2}, T_{\delta_2}, \dots, T_{\delta_{k-1}}) = 0
    .\]
    Applying the lantern relation \eqref{eq:lantern}, it suffices
    to show that 
    \[
    \A(T_{\gamma_1}, T_{\varepsilon}^2 T_{\gamma_1} T_{\gamma_3}
    T_{\varepsilon''}^{-1} T_{\varepsilon'}^{-1}, T_{\delta_2}, \dots, T_{\delta_{k-1}}) + 
    \A(T_{\gamma_3}, T_{\varepsilon}^2 T_{\gamma_1} T_{\gamma_3}
    T_{\varepsilon''}^{-1} T_{\varepsilon'}^{-1}, T_{\delta_2}, \dots,
    T_{\delta_{k-1}}) = 0
    .\]
    This is equivalent to
    \[
    \A(T_{\gamma_1}, T_{\varepsilon}^2 T_{\varepsilon''}^{-1} T_{\varepsilon'}^{-1}, T_{\delta_2}, \dots, T_{\delta_{k-1}}) + 
    \A(T_{\gamma_3}, T_{\varepsilon}^2 T_{\varepsilon''}^{-1} T_{\varepsilon'}^{-1}, T_{\delta_2}, \dots,
    T_{\delta_{k-1}}) = 0
    .\]
    By the BP-relation, each summand vanishes, and the lemma follows.
\end{proof}

We now consider the following two important consequences of
Lemma~\ref{lem:sum-relation-1}.

\begin{proposition} \label{prop:ab-cyc-ord-gen-3}
    Let \(\delta_1\) and \(\delta_2\) be disjoint, nonisotopic, separating
    simple closed curves on \(S_3\). Then in \(H_2(\I_3)\) we have
    \[
    2 \A(T_{\delta_1}, T_{\delta_2}) = 0
    .\]
\end{proposition}

\begin{proof}
    Choose a simple closed curve \(\delta_3\) on \(S_3\) such that \(\delta_1
    \cup \delta_2 \cup \delta_3\) bounds a pair of pants. Applying
    Lemma~\ref{lem:sum-relation-1} first to \(\delta_1, \delta_2, \delta_3\)
    and then to \(\delta_1, \delta_3, \delta_2\), we obtain
    \begin{align*}
        \A(T_{\delta_1}, T_{\delta_2}) + \A(T_{\delta_3}, T_{\delta_2}) &=
        \A(T_{\delta_2}, T_{\delta_2}), \\
        \A(T_{\delta_1}, T_{\delta_2}) + \A(T_{\delta_2}, T_{\delta_2}) &=
        \A(T_{\delta_3}, T_{\delta_2})
    .\end{align*}
    Summing these two relations yields the desired equality.
\end{proof}

\begin{remark}
We have the following splitting:
\[
H_{*}(\I_g; \Z[1/2]) =
H_{*}(\I_g; \Z[1/2])^{+} \oplus
H_{*}(\I_g; \Z[1/2])^{-},
\]
where \(H_{*}(\I_g; \Z[1/2])^{\pm}\) are the
eigenspaces of the action of the hyperelliptic involution.

Hain~\cite[Corollary~18]{hain2002} proved that \(H_2(\I_3; \Z[1/2])^{+} = 0\).
Moreover, the images of all abelian cycles \(\A(T_{\delta_1}, T_{\delta_2})\),
where \(\delta_1\) and \(\delta_2\) are disjoint separating simple closed curves,
lie in \(H_2(\I_3; \Z[1/2])^{+}\).
It follows that, in \(H_2(\I_3)\), each abelian cycle
\(\A(T_{\delta_1}, T_{\delta_2})\) has order \(2^{\ell}\) for some \(\ell \in \N\).
\end{remark}

From Corollary~\ref{cor:z2-vector-space} and
Proposition~\ref{prop:ab-cyc-ord-gen-3} it follows that
\(H_2^{\mathrm{ab,sep}}(\I_g)\) is a \(\Z/2\Z\)-vector space for \(g \geq
3\).

\begin{proposition} \label{prop:ab-sep-gen}
    For \(g \geq 3\), the \(\Z/2\Z\)-vector space
    \(H_2^{\mathrm{ab,sep}}(\I_g)\) is generated by abelian cycles
    \(\A(T_{\delta_1}, T_{\delta_2})\), where \(\delta_1\) and \(\delta_2\)
    are disjoint, nonisotopic, separating simple closed curves of genus \(1\)
    on \(S_g\).
\end{proposition}

\begin{proof}
    We show that for nonisotopic separating simple closed curves \(\delta_1\)
    and \(\delta_2\) on \(S_g\), the abelian cycle \(\A(T_{\delta_1},
    T_{\delta_2})\) can be expressed as a sum of abelian cycles
    \(\A(T_{\gamma_i}, T_{\theta_j})\), where \(\gamma_i\) and \(\theta_j\)
    are disjoint, nonisotopic, separating simple closed curves of genus \(1\)
    for all \(i\) and \(j\).

    Suppose that the curves \(\delta_1\) and \(\delta_2\) bound subsurfaces
    \(\Sigma_1, \Sigma_2, \Sigma_3\) of \(S_g\) such that \(\partial \Sigma_1
    = \delta_1\) and \(\partial \Sigma_2 = \delta_2\). Choose separating
    simple closed curves \(\gamma_1\) and \(\eta_1\) on \(\Sigma_1\) (see
    Figure~\ref{fig:gen-1-generation}) such that:
    \begin{itemize}
        \item \(\gamma_1\) has genus \(1\); and
        \item \(\gamma_1 \cup \eta_1 \cup \delta_1\) bounds a pair of pants.
    \end{itemize}
    By Lemma~\ref{lem:sum-relation-1} we have 
    \[
    \A(T_{\delta_1}, T_{\delta_2}) = \A(T_{\gamma_1}, T_{\delta_2}) +
    \A(T_{\eta_1}, T_{\delta_2})
    .\] 

    \begin{figure}[H]
        \centering
    \def\svgwidth{.6\columnwidth}
    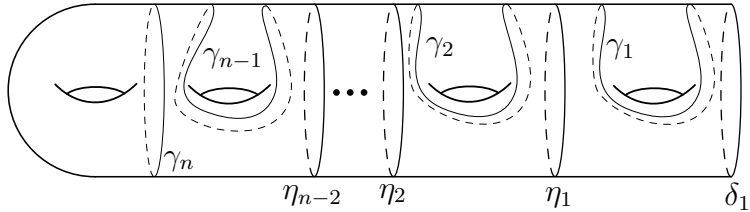

        \caption{The surface \(\Sigma_1\) and the curves on \(\Sigma_1\).}
        \label{fig:gen-1-generation}
    \end{figure}

    Continuing similarly with \(\A(T_{\eta_1}, T_{\delta_2})\) and repeating
    this process, we obtain a collection of separating simple closed curves
    \(\gamma_1, \dots, \gamma_n\) of genus \(1\) on \(\Sigma_1\) such that
    \[
    \A(T_{\delta_1}, T_{\delta_2}) = 
    \A(T_{\gamma_1}, T_{\delta_2}) + \cdots + \A(T_{\gamma_n}, T_{\delta_2})
    .\] 
    Applying the same argument to each abelian cycle \(\A(T_{\gamma_j},
    T_{\delta_2})\), we obtain a collection of separating simple closed curves
    \(\theta_1, \dots, \theta_m\) of genus \(1\) on \(\Sigma_2\) such that
    \[
    \A(T_{\gamma_j}, T_{\delta_2}) = \A(T_{\gamma_j}, T_{\theta_1}) + \cdots +
    \A(T_{\gamma_j}, T_{\theta_m}) 
    .\] 
    Hence, we obtain the following decomposition, which proves the
    proposition:
    \begin{align*}
        \A(T_{\delta_1}, T_{\delta_2}) = \sum_{i=1}^{n} \A(T_{\gamma_i},
        T_{\delta_2}) = \sum_{i=1}^{n} \sum_{j=1}^{m}  \A(T_{\gamma_i},
        T_{\theta_j})
    .\end{align*}
\end{proof}

We now state the second auxiliary relation in \(H_k(\I_g)\).

\begin{lemma} \label{lem:sum-relation-2}
    Consider separating simple closed curves \(\gamma_1, \gamma_2, \gamma_3\)
    and \(\delta_1, \dots, \delta_{k-1}\) on \(S_g\) such that:
    \begin{enumerate}
        \item the curves \(\gamma_1, \gamma_2, \gamma_3\) decompose \(S_g\)
        into a pair of pants and three subsurfaces \(\Sigma_1, \Sigma_2,
        \Sigma_3\), where \(\partial \Sigma_j = \gamma_j\) for \(j = 1, 2,
        3\);
        \item the curves \(\delta_1, \dots, \delta_{k-1}\) are pairwise
        disjoint, pairwise nonisotopic, and lie in \(\Sigma_1, \Sigma_2,
        \Sigma_3\);

        \item there exists a (possibly nonessential) separating simple closed
        curve \(\theta\) such that \(\gamma_3 \cup \theta\) bounds a
        subsurface \(\Sigma\) of genus \(1\), and the curves
        \(\delta_1, \dots, \delta_{k-1}\) are disjoint from
        \(\Sigma\).
    \end{enumerate}
    Then in \(H_k(\I_g)\) we have
    \[
    \A(T_{\gamma_1}, T_{\delta_1}, \dots, T_{\delta_{k-1}}) + 
    \A(T_{\gamma_2}, T_{\delta_1}, \dots, T_{\delta_{k-1}}) = 
    \A(T_{\gamma_3}, T_{\delta_1}, \dots, T_{\delta_{k-1}})
    .\] 
\end{lemma}

\begin{remark}
    Note the difference in the order of \(\gamma_1, \gamma_2, \gamma_3\) in
    the relations of Lemmas~\ref{lem:sum-relation-1}
    and~\ref{lem:sum-relation-2}. This difference is the key idea that will be
    used in the proof of Lemma~\ref{lem:gen-2-relation} below.
\end{remark}

\begin{proof}
    Similar to the proof of Lemma~\ref{lem:sum-relation-1}, we first consider
    the case in which none of the curves \(\delta_1, \dots, \delta_{k-1}\) is
    isotopic to \(\gamma_3\).

    Let \(\alpha, \alpha', \beta, \beta'\) be curves on \(S_g\), arranged as
    shown in Figure~\ref{fig:sum-relation-2}. Applying the lantern relation,
    we obtain
    \[
    T_{\alpha'} T_{\alpha} T_{\gamma_2} T_{\gamma_1} = T_{\gamma_{3}}
    T_{\beta} T_{\beta'}
    .\] 

    \begin{figure}[H]
        \centering
    \def\svgwidth{.9\columnwidth}
    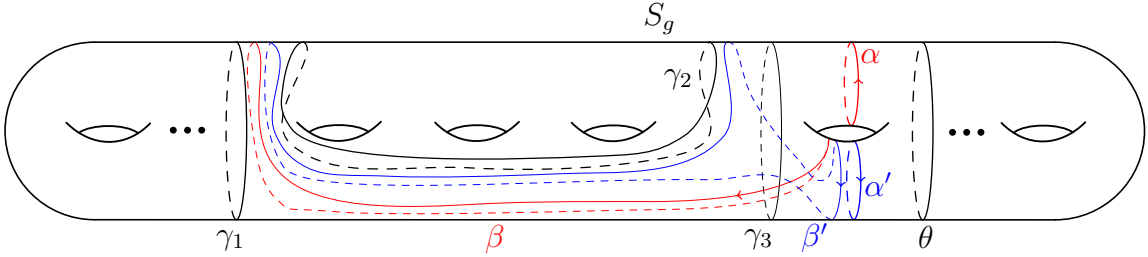

        \caption{The curves \(\alpha, \alpha', \beta, \beta'\) on \(S_g\).}
        \label{fig:sum-relation-2}
    \end{figure}
    As in the proof of Lemma~\ref{lem:sum-relation-1}, we have
    \(\HH_{\beta, \alpha} = \HH_{\alpha', \beta'} = (\HH_1, \HH_2)\), where
    \[
    \HH_1 = H_1(\Sigma_1) \oplus H_1(\Sigma_{\theta}) \oplus \langle x
    \rangle, \qquad \HH_2 = H_1(\Sigma_2) \oplus \langle x \rangle
    .\]
    Here, \(x = [\alpha] = [\alpha'] = [\beta] = [\beta']\), and
    \(\Sigma_{\theta}\) is the subsurface bounded by \(\theta\) lying to the
    right of \(\theta\) in Figure~\ref{fig:sum-relation-2}.
    By the BP-relation, we then obtain

    \begin{align*}
        \A(T_{\gamma_1}, T_{\delta_1}, \dots, T_{\delta_{k-1}}) + \A(T_{\gamma_2}, T_{\delta_1}, \dots,
        T_{\delta_{k-1}}) &= 
        \A(T_{\gamma_1} T_{\gamma_2}, T_{\delta_1}, \dots, T_{\delta_{k-1}})
        \\
        &= 
        \A(T_{\gamma_3} T_{\beta} T_{\beta'} T_{\alpha}^{-1}
        T_{\alpha'}^{-1}, T_{\delta_1}, \dots, T_{\delta_{k-1}}) \\
        &= \A(T_{\gamma_{3}}, T_{\delta_1}, \dots,
        T_{\delta_{k-1}}) \\
        &\qquad  \qquad +
        \A(T_{\beta} T_{\alpha}^{-1}, T_{\delta_1}\dots, T_{\delta_{k-1}}) \\
        & \qquad \qquad - 
        \A(T_{\alpha'} T_{\beta'}^{-1}, T_{\delta_1}\dots, T_{\delta_{k-1}}) 
        \\
        &= \A(T_{\gamma_{3}}, T_{\delta_1}, \dots,
        T_{\delta_{k-1}})
    ,\end{align*}
    and the lemma follows.
    
    The case in which one of the curves \(\delta_1, \dots, \delta_{k-1}\) is
    isotopic to \(\gamma_3\) is treated in the same way as in the proof of
    Lemma~\ref{lem:sum-relation-1}.
\end{proof}

\begin{remark}
    For \(k = 2\), Lemmas~\ref{lem:sum-relation-1} and~\ref{lem:sum-relation-2}
    are special cases of the more general Proposition~\ref{prop:key-relation},
    which will be proved in Section~\ref{sec:BCJ}. 
\end{remark}

An important consequence of Lemmas~\ref{lem:sum-relation-1}
and~\ref{lem:sum-relation-2} is the following lemma, which will play a key
role in Section~\ref{sec:order-2-hk}.

\begin{lemma}\label{lem:gen-2-relation}
    Let \(\delta_1,\dots,\delta_k\) be pairwise disjoint, pairwise nonisotopic,
    separating simple closed curves on \(S_g\).
    Assume that one of the following holds:
    \begin{enumerate}
        \item there exists a subsurface \(\Sigma\) of genus at least \(2\)
        among the subsurfaces \(\Sigma_1, \dots, \Sigma_{k+1}\)
        obtained by cutting \(S_g\) along \(\delta_1, \dots, \delta_k\);
        
        \item there exists a curve \(\delta_j\) that bounds a subsurface
        \(\Sigma\) of genus at least \(2\) containing exactly one other curve
        \(\delta_i\), and all remaining curves \(\delta_\ell\) are disjoint
        from \(\Sigma\).
    \end{enumerate}
    Then we have
    \[
    2\A(T_{\delta_1}, \dots, T_{\delta_k}) = 0
    .\]
\end{lemma}

\begin{proof}
    We first consider the case in which condition~(2) holds. We may assume,
    without loss of generality, that
    \begin{itemize}
        \item \(\delta_1\) bounds a subsurface \(\Sigma\) of genus at least
        \(2\);
        \item \(\delta_2\) lies inside \(\Sigma\); and
        \item the remaining curves \(\delta_3, \dots, \delta_k\) are disjoint
        from \(\Sigma\).
    \end{itemize}

    We begin with the case \(g(\Sigma) \geq 3\). By
    Proposition~\ref{prop:johnson3}, we have \(T_{\delta_2}^2 \in [\I(\Sigma),
    \I(\Sigma)]\). Since each \(T_{\delta_j}\) commutes with all elements of
    \(\I(\Sigma)\), it follows that
    \[
    2 \A(T_{\delta_1}, \dots, T_{\delta_k}) = 0
    .\] 

    We now turn to the case \(g(\Sigma) = 2\). In this case, \(\delta_2\)
    bounds a subsurface of genus~\(1\). Choose a separating simple closed
    curve \(\gamma\) of genus~\(1\) on \(\Sigma\), disjoint from \(\delta_2\).
    Apply Lemma~\ref{lem:sum-relation-1} first to \(\delta_2, \gamma,
    \delta_1\), and then to \(\gamma, \delta_2, \delta_1\) to obtain
    \begin{align*}
        \A(T_{\delta_2}, T_{\delta_2}, \dots, T_{\delta_k}) + \A(T_{\delta_1},
        \dots, T_{\delta_k}) &= \A(T_{\gamma}, T_{\delta_2}, \dots,
        T_{\delta_k}), \\
        \A(T_{\gamma}, T_{\delta_2}, \dots, T_{\delta_k}) + \A(T_{\delta_1},
        \dots, T_{\delta_k}) &= \A(T_{\delta_2}, T_{\delta_2}, \dots,
        T_{\delta_k})
    .\end{align*}
    Summing these two relations yields the desired equality.

    Next, we consider case~(1). Without loss of generality, we may assume that
    \(\delta_1, \dots, \delta_n\) (with \(n \leq k\)) bound a subsurface
    \(\Sigma\) of genus at least \(2\) that does not contain the curves
    \(\delta_{n+1}, \dots, \delta_k\).
    Choose separating simple closed curves \(\gamma_1\) and \(\gamma_2\) on
    \(\Sigma\) such that
    \begin{itemize}
        \item \(\gamma_1\) has genus~\(1\); and
        \item \(\gamma_1, \gamma_2\), and \(\delta_1\) bound a pair of pants
        and they are arranged as shown in
        Figure~\ref{fig:gen2-rel-many-boundaries}.
    \end{itemize}

    \begin{figure}[H]
        \centering
    \def\svgwidth{.4\columnwidth}
    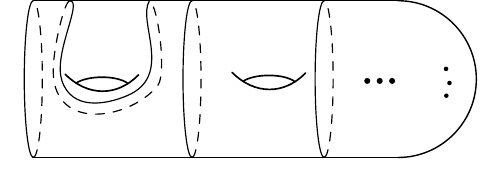

        \caption{The surface \(\Sigma\) and the curves \(\gamma_1, \gamma_2, \theta\) on \(\Sigma\).}
        \label{fig:gen2-rel-many-boundaries}
    \end{figure}
    
    In Figure~\ref{fig:gen2-rel-many-boundaries}, the curves \(\delta_2,
    \dots, \delta_n\) lie to the right of \(\gamma_2\).  We can, however,
    choose a separating simple closed curve \(\theta\) such that \(\delta_2,
    \dots, \delta_n\) lie to the right of \(\theta\), as illustrated in the
    same figure.

    Apply Lemma~\ref{lem:sum-relation-1} to \(\delta_1, \gamma_1, \gamma_2\) 
    to obtain
    \[
    \A(T_{\delta_1}, \dots, T_{\delta_k}) + \A(T_{\gamma_2}, T_{\delta_2},
    \dots, T_{\delta_k}) = \A(T_{\gamma_1}, T_{\delta_2}, \dots,
    T_{\delta_k})
    .\]
    The choice of \(\theta\) allows us to apply
    Lemma~\ref{lem:sum-relation-2} to \(\delta_1, \gamma_1,
    \gamma_2\), giving
    \[
    \A(T_{\delta_1}, \dots, T_{\delta_k}) + \A(T_{\gamma_1}, T_{\delta_2},
    \dots, T_{\delta_k}) = \A(T_{\gamma_2}, T_{\delta_2}, \dots,
    T_{\delta_k})
    .\] 
    Adding these two relations yields the desired equality.
\end{proof}

%% file: figures/sum-relation-1.pdf_tex
\begingroup%
  \makeatletter%
  \providecommand\color[2][]{%
    \errmessage{(Inkscape) Color is used for the text in Inkscape, but the package 'color.sty' is not loaded}%
    \renewcommand\color[2][]{}%
  }%
  \providecommand\transparent[1]{%
    \errmessage{(Inkscape) Transparency is used (non-zero) for the text in Inkscape, but the package 'transparent.sty' is not loaded}%
    \renewcommand\transparent[1]{}%
  }%
  \providecommand\rotatebox[2]{#2}%
  \newcommand*\fsize{\dimexpr\f@size pt\relax}%
  \newcommand*\lineheight[1]{\fontsize{\fsize}{#1\fsize}\selectfont}%
  \ifx\svgwidth\undefined%
    \setlength{\unitlength}{165.02625112bp}%
    \ifx\svgscale\undefined%
      \relax%
    \else%
      \setlength{\unitlength}{\unitlength * \real{\svgscale}}%
    \fi%
  \else%
    \setlength{\unitlength}{\svgwidth}%
  \fi%
  \global\let\svgwidth\undefined%
  \global\let\svgscale\undefined%
  \makeatother%
  \begin{picture}(1,0.52171918)%
    \lineheight{1}%
    \setlength\tabcolsep{0pt}%
    \put(0,0){\includegraphics[width=\unitlength,page=1]{sum-relation-1.pdf}}%
    \put(0.38477197,0.43268693){\color[rgb]{0,0,0}\makebox(0,0)[t]{\lineheight{1.25}\smash{\begin{tabular}[t]{c}$\gamma_2$\end{tabular}}}}%
    \put(0,0){\includegraphics[width=\unitlength,page=2]{sum-relation-1.pdf}}%
    \put(0.54172023,0.07990026){\color[rgb]{1,0,0}\makebox(0,0)[t]{\lineheight{1.25}\smash{\begin{tabular}[t]{c}$\varepsilon'$\end{tabular}}}}%
    \put(0.45124044,0.47430226){\color[rgb]{0.65490196,0,1}\makebox(0,0)[t]{\lineheight{1.25}\smash{\begin{tabular}[t]{c}$\varepsilon$\end{tabular}}}}%
    \put(0,0){\includegraphics[width=\unitlength,page=3]{sum-relation-1.pdf}}%
    \put(0.31861471,0.15627651){\color[rgb]{0,0,0.97647059}\makebox(0,0)[t]{\lineheight{1.25}\smash{\begin{tabular}[t]{c}$\varepsilon''$\end{tabular}}}}%
    \put(0,0){\includegraphics[width=\unitlength,page=4]{sum-relation-1.pdf}}%
    \put(0.2239662,0.008582){\color[rgb]{0,0,0}\makebox(0,0)[t]{\lineheight{1.25}\smash{\begin{tabular}[t]{c}$\gamma_1$\end{tabular}}}}%
    \put(0.76137737,0.00668073){\color[rgb]{0,0,0}\makebox(0,0)[t]{\lineheight{1.25}\smash{\begin{tabular}[t]{c}$\gamma_3$\end{tabular}}}}%
  \end{picture}%
\endgroup%

%% file: figures/gen-1-generation.pdf_tex
\begingroup%
  \makeatletter%
  \providecommand\color[2][]{%
    \errmessage{(Inkscape) Color is used for the text in Inkscape, but the package 'color.sty' is not loaded}%
    \renewcommand\color[2][]{}%
  }%
  \providecommand\transparent[1]{%
    \errmessage{(Inkscape) Transparency is used (non-zero) for the text in Inkscape, but the package 'transparent.sty' is not loaded}%
    \renewcommand\transparent[1]{}%
  }%
  \providecommand\rotatebox[2]{#2}%
  \newcommand*\fsize{\dimexpr\f@size pt\relax}%
  \newcommand*\lineheight[1]{\fontsize{\fsize}{#1\fsize}\selectfont}%
  \ifx\svgwidth\undefined%
    \setlength{\unitlength}{329.33322216bp}%
    \ifx\svgscale\undefined%
      \relax%
    \else%
      \setlength{\unitlength}{\unitlength * \real{\svgscale}}%
    \fi%
  \else%
    \setlength{\unitlength}{\svgwidth}%
  \fi%
  \global\let\svgwidth\undefined%
  \global\let\svgscale\undefined%
  \makeatother%
  \begin{picture}(1,0.26966206)%
    \lineheight{1}%
    \setlength\tabcolsep{0pt}%
    \put(0,0){\includegraphics[width=\unitlength,page=1]{gen-1-generation.pdf}}%
    \put(0.96010115,0.00210424){\color[rgb]{0,0,0}\makebox(0,0)[t]{\lineheight{1.25}\smash{\begin{tabular}[t]{c}$\delta_1$\end{tabular}}}}%
    \put(0,0){\includegraphics[width=\unitlength,page=2]{gen-1-generation.pdf}}%
    \put(0.80629655,0.20369597){\color[rgb]{0,0,0}\makebox(0,0)[t]{\lineheight{1.25}\smash{\begin{tabular}[t]{c}$\gamma_1$\end{tabular}}}}%
    \put(0.56909792,0.20678723){\color[rgb]{0,0,0}\makebox(0,0)[t]{\lineheight{1.25}\smash{\begin{tabular}[t]{c}$\gamma_2$\end{tabular}}}}%
    \put(0.22590099,0.05617836){\color[rgb]{0,0,0}\makebox(0,0)[t]{\lineheight{1.25}\smash{\begin{tabular}[t]{c}$\gamma_n$\end{tabular}}}}%
    \put(0.29674799,0.19237837){\color[rgb]{0,0,0}\makebox(0,0)[t]{\lineheight{1.25}\smash{\begin{tabular}[t]{c}$\gamma_{n-1}$\end{tabular}}}}%
    \put(0.72524555,0.00776928){\color[rgb]{0,0,0}\makebox(0,0)[t]{\lineheight{1.25}\smash{\begin{tabular}[t]{c}$\eta_1$\end{tabular}}}}%
    \put(0.506622,0.01232401){\color[rgb]{0,0,0}\makebox(0,0)[t]{\lineheight{1.25}\smash{\begin{tabular}[t]{c}$\eta_2$\end{tabular}}}}%
    \put(0.40083844,0.01070069){\color[rgb]{0,0,0}\makebox(0,0)[t]{\lineheight{1.25}\smash{\begin{tabular}[t]{c}$\eta_{n-2}$\end{tabular}}}}%
  \end{picture}%
\endgroup%

%% file: figures/sum-relation-2.pdf_tex
\begingroup%
  \makeatletter%
  \providecommand\color[2][]{%
    \errmessage{(Inkscape) Color is used for the text in Inkscape, but the package 'color.sty' is not loaded}%
    \renewcommand\color[2][]{}%
  }%
  \providecommand\transparent[1]{%
    \errmessage{(Inkscape) Transparency is used (non-zero) for the text in Inkscape, but the package 'transparent.sty' is not loaded}%
    \renewcommand\transparent[1]{}%
  }%
  \providecommand\rotatebox[2]{#2}%
  \newcommand*\fsize{\dimexpr\f@size pt\relax}%
  \newcommand*\lineheight[1]{\fontsize{\fsize}{#1\fsize}\selectfont}%
  \ifx\svgwidth\undefined%
    \setlength{\unitlength}{347.78988575bp}%
    \ifx\svgscale\undefined%
      \relax%
    \else%
      \setlength{\unitlength}{\unitlength * \real{\svgscale}}%
    \fi%
  \else%
    \setlength{\unitlength}{\svgwidth}%
  \fi%
  \global\let\svgwidth\undefined%
  \global\let\svgscale\undefined%
  \makeatother%
  \begin{picture}(1,0.25218216)%
    \lineheight{1}%
    \setlength\tabcolsep{0pt}%
    \put(0,0){\includegraphics[width=\unitlength,page=1]{sum-relation-2.pdf}}%
    \put(0.07690178,0.00581164){\color[rgb]{0,0,0}\makebox(0,0)[t]{\lineheight{1.25}\smash{\begin{tabular}[t]{c}$\gamma_1$\end{tabular}}}}%
    \put(0,0){\includegraphics[width=\unitlength,page=2]{sum-relation-2.pdf}}%
    \put(0.61788004,0.20193668){\color[rgb]{0,0,0}\makebox(0,0)[t]{\lineheight{1.25}\smash{\begin{tabular}[t]{c}$\gamma_2$\end{tabular}}}}%
    \put(0.71504653,0.00708615){\color[rgb]{0,0,0}\makebox(0,0)[t]{\lineheight{1.25}\smash{\begin{tabular}[t]{c}$\gamma_3$\end{tabular}}}}%
    \put(0,0){\includegraphics[width=\unitlength,page=3]{sum-relation-2.pdf}}%
    \put(0.8504966,0.22569259){\color[rgb]{1,0,0}\makebox(0,0)[t]{\lineheight{1.25}\smash{\begin{tabular}[t]{c}$\alpha$\end{tabular}}}}%
    \put(0,0){\includegraphics[width=\unitlength,page=4]{sum-relation-2.pdf}}%
    \put(0.85973402,0.06428459){\color[rgb]{0,0,1}\makebox(0,0)[t]{\lineheight{1.25}\smash{\begin{tabular}[t]{c}$\alpha'$\end{tabular}}}}%
    \put(0,0){\includegraphics[width=\unitlength,page=5]{sum-relation-2.pdf}}%
    \put(0.39627195,0.00205206){\color[rgb]{1,0,0}\makebox(0,0)[t]{\lineheight{1.25}\smash{\begin{tabular}[t]{c}$\beta$\end{tabular}}}}%
    \put(0,0){\includegraphics[width=\unitlength,page=6]{sum-relation-2.pdf}}%
    \put(0.78284613,0.00261444){\color[rgb]{0,0,1}\makebox(0,0)[t]{\lineheight{1.25}\smash{\begin{tabular}[t]{c}$\beta'$\end{tabular}}}}%
    \put(0,0){\includegraphics[width=\unitlength,page=7]{sum-relation-2.pdf}}%
    \put(0.91721365,0.00075044){\color[rgb]{0,0,0}\makebox(0,0)[t]{\lineheight{1.25}\smash{\begin{tabular}[t]{c}$\theta$\end{tabular}}}}%
    \put(0,0){\includegraphics[width=\unitlength,page=8]{sum-relation-2.pdf}}%
    \put(3.56268834,-0.03179268){\color[rgb]{0,0,0}\makebox(0,0)[t]{\lineheight{1.25}\smash{\begin{tabular}[t]{c}$S_g$\end{tabular}}}}%
    \put(0,0){\includegraphics[width=\unitlength,page=9]{sum-relation-2.pdf}}%
    \put(3.04421068,-0.2982512){\color[rgb]{0,0,0}\makebox(0,0)[t]{\lineheight{1.25}\smash{\begin{tabular}[t]{c}$\gamma_1$\end{tabular}}}}%
    \put(0,0){\includegraphics[width=\unitlength,page=10]{sum-relation-2.pdf}}%
    \put(3.5851889,-0.10212616){\color[rgb]{0,0,0}\makebox(0,0)[t]{\lineheight{1.25}\smash{\begin{tabular}[t]{c}$\gamma_2$\end{tabular}}}}%
    \put(3.68235536,-0.2969767){\color[rgb]{0,0,0}\makebox(0,0)[t]{\lineheight{1.25}\smash{\begin{tabular}[t]{c}$\gamma_3$\end{tabular}}}}%
    \put(0,0){\includegraphics[width=\unitlength,page=11]{sum-relation-2.pdf}}%
    \put(3.81780543,-0.07837026){\color[rgb]{1,0,0}\makebox(0,0)[t]{\lineheight{1.25}\smash{\begin{tabular}[t]{c}$\alpha$\end{tabular}}}}%
    \put(0,0){\includegraphics[width=\unitlength,page=12]{sum-relation-2.pdf}}%
    \put(3.82704284,-0.23977826){\color[rgb]{0,0,1}\makebox(0,0)[t]{\lineheight{1.25}\smash{\begin{tabular}[t]{c}$\alpha'$\end{tabular}}}}%
    \put(0,0){\includegraphics[width=\unitlength,page=13]{sum-relation-2.pdf}}%
    \put(3.36358075,-0.30201079){\color[rgb]{1,0,0}\makebox(0,0)[t]{\lineheight{1.25}\smash{\begin{tabular}[t]{c}$\beta$\end{tabular}}}}%
    \put(0,0){\includegraphics[width=\unitlength,page=14]{sum-relation-2.pdf}}%
    \put(3.75015501,-0.30144841){\color[rgb]{0,0,1}\makebox(0,0)[t]{\lineheight{1.25}\smash{\begin{tabular}[t]{c}$\beta'$\end{tabular}}}}%
    \put(0,0){\includegraphics[width=\unitlength,page=15]{sum-relation-2.pdf}}%
    \put(3.8845226,-0.30331241){\color[rgb]{0,0,0}\makebox(0,0)[t]{\lineheight{1.25}\smash{\begin{tabular}[t]{c}$\theta$\end{tabular}}}}%
  \end{picture}%
\endgroup%

%% file: figures/gen2-rel-many-boundaries.pdf_tex
\begingroup%
  \makeatletter%
  \providecommand\color[2][]{%
    \errmessage{(Inkscape) Color is used for the text in Inkscape, but the package 'color.sty' is not loaded}%
    \renewcommand\color[2][]{}%
  }%
  \providecommand\transparent[1]{%
    \errmessage{(Inkscape) Transparency is used (non-zero) for the text in Inkscape, but the package 'transparent.sty' is not loaded}%
    \renewcommand\transparent[1]{}%
  }%
  \providecommand\rotatebox[2]{#2}%
  \newcommand*\fsize{\dimexpr\f@size pt\relax}%
  \newcommand*\lineheight[1]{\fontsize{\fsize}{#1\fsize}\selectfont}%
  \ifx\svgwidth\undefined%
    \setlength{\unitlength}{233.53968583bp}%
    \ifx\svgscale\undefined%
      \relax%
    \else%
      \setlength{\unitlength}{\unitlength * \real{\svgscale}}%
    \fi%
  \else%
    \setlength{\unitlength}{\svgwidth}%
  \fi%
  \global\let\svgwidth\undefined%
  \global\let\svgscale\undefined%
  \makeatother%
  \begin{picture}(1,0.37886434)%
    \lineheight{1}%
    \setlength\tabcolsep{0pt}%
    \put(0,0){\includegraphics[width=\unitlength,page=1]{gen2-rel-many-boundaries.pdf}}%
    \put(0.38394276,0.00472094){\color[rgb]{0,0,0}\makebox(0,0)[t]{\lineheight{1.25}\smash{\begin{tabular}[t]{c}$\gamma_2$\end{tabular}}}}%
    \put(0.05626456,0.0029674){\color[rgb]{0,0,0}\makebox(0,0)[t]{\lineheight{1.25}\smash{\begin{tabular}[t]{c}$\delta_1$\end{tabular}}}}%
    \put(0.25644771,0.3146354){\color[rgb]{0,0,0}\makebox(0,0)[t]{\lineheight{1.25}\smash{\begin{tabular}[t]{c}$\gamma_1$\end{tabular}}}}%
    \put(0,0){\includegraphics[width=\unitlength,page=2]{gen2-rel-many-boundaries.pdf}}%
    \put(0.65740546,0.00111755){\color[rgb]{0,0,0}\makebox(0,0)[t]{\lineheight{1.25}\smash{\begin{tabular}[t]{c}$\theta$\end{tabular}}}}%
    \put(0.80741051,0.0886995){\color[rgb]{0,0,0}\makebox(0,0)[t]{\lineheight{1.25}\smash{\begin{tabular}[t]{c}$\delta_n$\end{tabular}}}}%
    \put(0.82137191,0.3192252){\color[rgb]{0,0,0}\makebox(0,0)[t]{\lineheight{1.25}\smash{\begin{tabular}[t]{c}$\delta_2$\end{tabular}}}}%
  \end{picture}%
\endgroup%

%% file: sections/BCJ.tex
\section{Birman--Craggs--Johnson homomorphism and relations in
\(H_2(\I_g)\)} \label{sec:BCJ}

To state the main result of this section, we first recall the
\emph{Birman--Craggs--Johnson homomorphism}
\(\sigma \colon \I_g \to \B_3'\)
(see \cite{johnsonBCJ} for details).
We briefly review its construction.

A basis \(a_1, \dots, a_g, b_1, \dots, b_g\) of the group \(H_1(S_g)\)
is called \emph{symplectic} if
\[
a_i \cdot b_j = \delta_{ij}, \quad a_i \cdot a_j = b_i \cdot b_j = 0,
\]
where \(\delta_{ij}\) denotes the Kronecker delta.
A subgroup \(V \subset H_1(S_g)\) is called \emph{a symplectic summand} if the
restriction of the intersection form to \(V\) is unimodular.
Equivalently, this means that \(H_1(S_g) = V \oplus V^{\perp}\),
where the decomposition is orthogonal with respect to the intersection form.
A symplectic basis of \(H_1(S_g; \Z/2\Z)\) and a symplectic subspace
of \(H_1(S_g; \Z/2\Z)\) are defined in the same way.

Recall that for integral homology classes \(x, y, \dots\), their reductions modulo~2
are denoted by bold letters \(\bx, \by, \dots\).

Let \(S = S_g^b\), where \(b \in \{0, 1\}\). 
Let \(\B(S)\) be the algebra over \(\Z/2\Z\) generated by formal variables
\(\overline{\bx}\) for all \(\bx \in H_1(S; \Z/2\Z)\), subject to the
following relations for all \(\bx, \by \in H_1(S; \Z/2\Z)\):
\begin{enumerate}
    \item \(\overline{\bx + \by} = \overline{\bx} + \overline{\by} + (\bx
    \cdot \by)\), where \(\bx \cdot \by\) denotes the \(\mathrm{mod}\; 2\) intersection number;
    \item \(\overline{\bx}^2 = \overline{\bx}\).
\end{enumerate}
We have the following description of \(\B(S)\).  From relations (1) and (2)
it follows that, for any basis \(\be_1, \dots, \be_{2g}\) of \(H_1(S;
\Z/2\Z)\), \(\B(S)\) is the algebra of Boolean polynomials in the formal
variables \(\overline{\be}_1, \dots, \overline{\be}_{2g}\).

The \emph{Arf invariant} is the quadratic polynomial
\[
\Arf = \sum_{j=1}^{g} \overline{\ba}_j \overline{\bb}_j
,\]
where \(\ba_1, \dots, \ba_g, \bb_1, \dots, \bb_g\) is a symplectic basis of
\(H_1(S; \Z/2\Z)\). It is well known that
\(\Arf\) is independent of the choice of symplectic basis.

Let \(\B'(S)\) denote the quotient algebra \(\B(S)/(\Arf)\). 
For each \(k\), let \(\B_k(S) \subset \B(S)\) and \(\B_k'(S) \subset \B'(S)\) 
denote the subspaces consisting of polynomials of degree at most \(k\). 
When the surface is clear from the context, we will simply write \(\B_k\) and
\(\B_k'\).

We denote by \(\sigma\) the Birman--Craggs--Johnson homomorphisms
\begin{align*}
    &\sigma \colon \I_g \to \B_3'(S_g), \\
    &\sigma \colon \I_g^1 \to \B_3(S_g^1)
.\end{align*}

Recall that the value of the Birman--Craggs--Johnson homomorphism on a Dehn
twist about a separating simple closed curve \(\delta\) is given by
\begin{align} \label{eq:BCJ-sep}
    \sigma(T_{\delta}) = \sum_{i=1}^{g'} \overline{\ba}_i \overline{\bb}_i
,\end{align}
where \(\delta\) bounds a subsurface \(\Sigma\) of genus \(g'\),
and \(\ba_1, \dots, \ba_{g'}, \bb_1, \dots, \bb_{g'}\) form a symplectic basis
of \(H_1(\Sigma; \Z/2\Z)\).
This formula is well defined for a closed surface \(S_g\), since
\(\Arf = 0\) in \(\B_3'(S_g)\).

Now we are in a position to state the main result of this section.

\begin{proposition} \label{prop:key-relation}
    Let \(\delta\) be a separating simple closed curve on \(S_g\) with \(g
    \geq 4\), bounding a subsurface \(\Sigma\) of genus at least \(3\).
    Then for separating simple closed curves \(\gamma_1, \dots, \gamma_n\) on
    \(\Sigma\), the following are equivalent:
    \begin{itemize}
        \item \(\sigma(T_{\gamma_1}) + \cdots + \sigma(T_{\gamma_n}) = 0\) or
        \(\sigma(T_{\gamma_1}) + \cdots + \sigma(T_{\gamma_n}) = \sigma(T_{\delta})\);
        \item \(\A(T_{\gamma_1}, T_{\delta}) + \cdots + \A(T_{\gamma_n}, T_{\delta}) = 0\)
        in \(H_2(\I_g)\).
    \end{itemize}
\end{proposition}

\begin{remark}
    Proposition~\ref{prop:key-relation} also holds without the assumption that
    \(g(\Sigma) \geq 3\). In particular, using
    Lemma~\ref{lem:sum-relation-1}, the case \(g(\Sigma) = 2\) can be deduced
    from the case \(g(\Sigma) \geq 3\). Moreover, the statement remains
    valid even if the curves \(\gamma_1, \dots, \gamma_n\) are not all
    required to lie on one side of \(\delta\). Since these extensions will not
    be needed, we omit the proof.
\end{remark}

To prove Proposition~\ref{prop:key-relation} and to understand its nature,
we require the following construction. Consider the pushforward homomorphism
\(\sigma_2 \colon H_2(\I_g) \to H_2(\B_3')\) induced by the Birman--Craggs--Johnson
homomorphism \(\sigma \colon \I_g \to \B_3'\).

Since \(\B_3'\) is an abelian group, we have an isomorphism
\(H_2(\B_3') \cong \wedge^2 \B_3'\), where \(\wedge^2 \B_3'\) is defined as
the quotient of \(\B_3' \otimes \B_3'\) by all relations of the form \(x
\otimes x = 0\).

It is easy to check, for example using the bar resolution, that on an
abelian cycle \(\A(h_1, h_2)\) with \(h_1, h_2 \in \I_g\), the homomorphism
\(\sigma_2\) acts as
\begin{align} \label{eq:abelian-cycle-boolean}
    \sigma_2(\A(h_1, h_2)) = \sigma(h_1) \wedge \sigma(h_2)
.\end{align}

From~\eqref{eq:BCJ-sep}, it follows that \(\sigma(T_{\delta}) \in \B_2'\) for
all separating simple closed curves \(\delta\) on \(S_g\). Restricting
\(\sigma_2\) to the subgroup \(H_2^{\mathrm{ab,sep}}(\I_g) \subset H_2(\I_g)\)
yields a homomorphism
\[
\sigma_2 \colon H_2^{\mathrm{ab,sep}}(\I_g) \to \wedge^2 \B_2'
.\]

It is straightforward to check that this is indeed a homomorphism of \(\Sp(2g,
\Z)\)-modules, since \(\sigma\) itself is such a homomorphism.
Consequently, we obtain an \(\Sp(2g, \Z)\)-equivariant homomorphism of
\(\Z/2\Z\)-vector spaces
\[
\sigma_2 \colon H_2^{\mathrm{ab,sep}}(\I_g) \to \wedge^2 \B_2'
,\]
and the value of \(\sigma_2\) on an abelian cycle \(\A(T_{\delta_1},
T_{\delta_2})\) is given by formula~\eqref{eq:abelian-cycle-boolean}.

\begin{remark}
    From this perspective, we can reformulate
    Proposition~\ref{prop:key-relation} as follows: the relation
    \[
    \A(T_{\gamma_1}, T_{\delta}) + \cdots + \A(T_{\gamma_n}, T_{\delta}) = 0
    \] 
    holds if and only if
    \[
    \sigma_2(\A(T_{\gamma_1}, T_{\delta}) + \cdots + \A(T_{\gamma_n}, T_{\delta})) = 0
    .\]
\end{remark}

In the proof of Proposition~\ref{prop:key-relation}, we will also require the
\emph{Johnson homomorphisms}, denoted by \(\tau\):
\begin{align*}
    &\tau \colon \I_g \to \wedge^3 H_1(S_g) / H_1(S_g),\\
    &\tau \colon \I_g^1 \to \wedge^3 H_1(S_g^1).
\end{align*}
Here, the inclusion \(H_1(S_g) \hookrightarrow \wedge^3 H_1(S_g)\) is given by
\(x \mapsto \Omega \wedge x\), where \(\Omega \in \wedge^2 H_1(S_g)\) denotes
the dual of the intersection form.
For the construction of the Johnson homomorphism and its basic properties,
see~\cite{johnson-ab}.
By~\cite[Lemma~4A]{johnson-ab}, it follows that \(\tau(T_{\delta}) = 0\) whenever
\(\delta\) is a separating simple closed curve.

Now we are ready to prove Proposition~\ref{prop:key-relation}.

\begin{proof}[Proof of Proposition~\ref{prop:key-relation}]
    Suppose \(\A(T_{\gamma_1}, T_{\delta}) + \cdots + \A(T_{\gamma_n},
    T_{\delta}) = 0\) in \(H_2(\I_g)\). Then we have
    \begin{align*}
        0 = \sigma_2(\A(T_{\gamma_1}, T_{\delta}) + \cdots + \A(T_{\gamma_n},
        T_{\delta})) &= \sigma_2(\A(T_{\gamma_1} \cdots T_{\gamma_n},
        T_{\delta})) \\
        &= \sigma(T_{\gamma_1} \cdots T_{\gamma_n}) \wedge \sigma(T_{\delta})
    .\end{align*}
    Since \(\sigma(T_{\delta}) \neq 0\) by 
    formula~\eqref{eq:BCJ-sep}, one of the following holds:
    \[
    \sigma(T_{\gamma_1} \cdots T_{\gamma_n}) = 0 \quad \text{or} \quad
    \sigma(T_{\gamma_1} \cdots T_{\gamma_n}) = \sigma(T_{\delta}),
    \]
    as required.

    We now prove the converse statement. We may assume that
    \(\sigma(T_{\gamma_1}) + \cdots + \sigma(T_{\gamma_n}) = 0\),
    since otherwise we could simply add the curve \(\gamma_{n+1} = \delta\) to
    the set \(\{\gamma_1, \dots, \gamma_n\}\); this does not affect the
    statement to be proved, as \(\A(T_{\delta}, T_{\delta}) = 0\) in
    \(H_2(\I_g)\).

    Let \(f = T_{\gamma_1} \cdots T_{\gamma_n}\). We need to show that the
    equality \(\sigma(f) = 0\) implies \(\A(f, T_{\delta}) = 0\) in
    \(H_2(\I_g)\). We consider the restrictions of the
    Birman--Craggs--Johnson and Johnson homomorphisms to the Torelli subgroups
    (see~\cite[Section 10]{johnsonBCJ} and~\cite[Section 6]{johnson-ab}):
    \begin{align*}
       \sigma|_{\Sigma} &= \sigma_{\Sigma} \colon \I(\Sigma) \to \B_3(\Sigma), \\
       \tau|_{\Sigma} &= \tau_{\Sigma} \colon \I(\Sigma) \to \wedge^3 H_1(\Sigma)
    .\end{align*}
    
    Since \(\sigma_{\Sigma}(f) = \tau_{\Sigma}(f) = 0\), it follows
    from~\cite[Theorem~3]{johnson3} that \(f \in [\I(\Sigma), \I(\Sigma)]\);
    that is,
    \[
    f = [g_1, h_1] \cdots [g_k, h_k]
    ,\]
    where \(g_i, h_i \in \I_g\) are supported on \(\Sigma\) for
    \(i = 1, \dots, k\).

    It then follows immediately from Fact~\ref{ft:lin} and
    Corollary~\ref{cor:comm} that \(\A(f, T_{\delta}) = 0\), since each
    \(g_i\) and \(h_i\) commutes with \(T_{\delta}\) for all \(i = 1, \dots,
    k\).
\end{proof}

%% file: sections/h2-is-finite-dimensional.tex
\section{Proof of Theorem~\ref{thm:main-2}} \label{sec:h2-fin-dim}

In this section, we prove that the vector space \(H_2^{\mathrm{ab,sep}}(\I_g)\)
is finite-dimensional for \(g \geq 4\). This follows from the fact
that there are only finitely many pairwise distinct abelian cycles
\(\A(T_{\delta_1}, T_{\delta_2})\), where \(\delta_1\) and \(\delta_2\) are
separating simple closed curves of genus~\(1\). We now state the precise
proposition.

\begin{proposition} \label{prop:injective-2}
    Let \(\delta_1, \delta_2\) and \(\gamma_1, \gamma_2\) be two pairs of
    separating simple closed curves of genus \(1\) on \(S_g\) with
    \(g \geq 4\), such that the curves in each pair are disjoint and
    nonisotopic.
    If \(\sigma(T_{\delta_1}) = \sigma(T_{\gamma_1})\) and
    \(\sigma(T_{\delta_2}) = \sigma(T_{\gamma_2})\), then
    \[
    \A(T_{\delta_1}, T_{\delta_2}) = \A(T_{\gamma_1}, T_{\gamma_2})
    .\]
\end{proposition}

\begin{remark}
From Proposition~\ref{prop:injective-2}, one can deduce the following:
if two abelian cycles \(\A(T_{\delta_1}, T_{\delta_2})\) and
\(\A(T_{\gamma_1}, T_{\gamma_2})\) satisfy
\[
\sigma_2(\A(T_{\delta_1}, T_{\delta_2})) =
\sigma_2(\A(T_{\gamma_1}, T_{\gamma_2}))
,\] 
then we have
\[
\A(T_{\delta_1}, T_{\delta_2}) = \A(T_{\gamma_1}, T_{\gamma_2})
.\] 
\end{remark}

\medskip
\begin{proof}[Proof of Theorem~\ref{thm:main-2}]
    Since \(\B_2'\) is a finite set, there are only finitely many
    possibilities for \(\sigma(T_{\delta_1})\) and \(\sigma(T_{\delta_2})\),
    and the theorem follows.
\end{proof}

To prove Proposition~\ref{prop:injective-2}, we first reduce the topological
problem to a linear-algebraic one by describing the connection between the
abelian cycles \(\A(T_{\delta_1}, T_{\delta_2})\) and the corresponding
splittings of \(H_1(S_g)\).

Consider disjoint, nonisotopic, separating simple closed curves \(\delta_1\)
and \(\delta_2\) on \(S_g\). Suppose that \(\delta_1 \cup \delta_2\)
decomposes \(S_g\) into three subsurfaces \(\Sigma_1, \Sigma, \Sigma_2\) such
that
\[
\partial \Sigma_1 = \delta_1, \qquad \partial \Sigma_2 = \delta_2, \qquad 
\partial \Sigma = \delta_1 \cup \delta_2
.\]
This induces a splitting
\begin{align*}
    H_1(S_g) = U_1 \oplus U \oplus U_2
,\end{align*}
where \(U_1 = H_1(\Sigma_1)\), \(U_2 = H_1(\Sigma_2)\), and \(U = H_1(\Sigma)\). 

We will also write the abelian cycle \(\A(T_{\delta_1}, T_{\delta_2})\) as
\(\A(U_1, U_2)\); this is well-defined by the following lemma.

\begin{lemma}
    The abelian cycle \(\A(T_{\delta_1}, T_{\delta_2})\) is determined by the
    pair \((U_1, U_2)\).
\end{lemma}

\begin{proof}
    Consider two other disjoint, nonisotopic, separating simple closed curves
    \(\delta_1'\) and \(\delta_2'\), inducing the same splitting of
    \(H_1(S_g)\). It follows that there exists an element \(f \in \I_g\)
    taking \(\delta_1\) to \(\delta_1'\) and \(\delta_2\) to \(\delta_2'\).
    Hence,
    \[
    \A(T_{\delta_1}, T_{\delta_2}) = f_{*} (\A(T_{\delta_1}, T_{\delta_2})) = 
    \A(T_{\delta_1'}, T_{\delta_2'})
    ,\] 
    and the lemma follows.
\end{proof}

Throughout this paper, by a \emph{splitting} of an abelian group \(H\) with a
symplectic form, we mean a decomposition
\[
H = X_1 \oplus \dots \oplus X_n,
\]
where \(X_1, \dots, X_n \subset H\) are nonzero, pairwise orthogonal symplectic
summands with respect to the symplectic form on \(H\).

From this perspective, we obtain the following consequence of the proof of
Proposition~\ref{prop:ab-cyc-ord-gen-3}.

\begin{corollary}
    Consider a splitting
    \[
    H_1(S_3) = U_1 \oplus U_2 \oplus U_3
    .\]
    Then
    \[
    \A(U_1, U_2) = \A(U_2, U_3) = \A(U_3, U_1).
    \]
    Therefore, in \(H_2(\I_3)\), the abelian cycle of this form is completely
    determined by the splitting and does not depend on the choice of summands.
\end{corollary}

\medskip

We say that a separating simple closed curve \(\gamma\) \emph{bounds} a
symplectic summand \(V \subset H_1(S_g)\) if it induces a decomposition \(H_1(S_g) = V
\oplus V^{\perp}\). 
It follows from~\eqref{eq:BCJ-sep} that the value \(\sigma(T_\gamma)\) depends only on the
unordered splitting \(\{V, V^{\perp}\}\). 
We define 
\[
\sigma(V) = \sigma(V^{\perp}) = \sigma(T_\gamma)
.\] 
Moreover, \(\sigma(V)\) depends only on \(\bV\) (the reduction of \(V\) modulo
\(2\)). We then set \(\sigma(\bV) = \sigma(V)\).

Consider two disjoint, nonisotopic, separating simple closed curves
\(\delta_1, \delta_2\) of genus \(1\) on \(S_g\). Let \(U_1, U_2 \subset
H_1(S_g)\) be the symplectic summands bounded by \(\delta_1\) and \(\delta_2\),
respectively. For \(j = 1,2\), let \(x_j, y_j\) form a symplectic basis of
\(U_j\). Then
\[
\sigma_2(\A(U_1, U_2)) = \sigma(T_{\delta_1}) \wedge \sigma(T_{\delta_2})
= \overline{\bx}_1 \overline{\by}_1 \wedge \overline{\bx}_2 \overline{\by}_2
.\]

\medskip
We now reformulate Proposition~\ref{prop:key-relation} in linear algebraic terms.

\begin{proposition} \label{prop:key-relation-lin}
    Let \(g \geq 4\), and let \(U \subset H_1(S_g)\) be a symplectic
    summand. Let \(V_1, \dots, V_n \subset H_1(S_g)\) be symplectic summands
    such that \(V_1, \dots, V_n \subset U\) and \(\rank U \geq 6\).
    Then the following are equivalent:
    \begin{itemize}
        \item \(\sigma(V_1) + \cdots + \sigma(V_n) = 0\) or
        \(\sigma(V_1) + \cdots + \sigma(V_n) = \sigma(U)\);
        \item \(\A(V_1, U^{\perp}) + \dots + \A(V_n, U^{\perp}) = 0\) in
        \(H_2(\I_g)\).
    \end{itemize}
\end{proposition}

\begin{remark}
    We will also use a variant of this proposition in which the arguments of
    the abelian cycle are interchanged.
\end{remark}

In the proof of Proposition~\ref{prop:injective-2}, we will require the
following lemmas.

\begin{lemma} \label{lem:mod-2-sigma}
    For \(g \geq 3\), let \(\bV \subset H_1(S_g; \Z/2\Z)\) be a
    \(2\)-dimensional symplectic subspace with \(\sigma(\bV) =
    \overline{\ba}\,\overline{\bb}\) for some \(\ba, \bb \in H_1(S_g;
    \Z/2\Z)\) satisfying \(\ba \cdot \bb = 1\). Then \(\bV = \langle \ba, \bb \rangle\).
\end{lemma}

\begin{proof}
    Let \(\bx, \by \in H_1(S_g; \Z/2\Z)\) form a symplectic basis for \(\bV\).
    Then \(\overline{\bx}\, \overline{\by} = \overline{\ba}\,
    \overline{\bb}\). We have the following decomposition
    \begin{align*}
        \bx &= \bx_0 + \bx_1, \qquad \text{with } \bx_0 \in \langle \ba, \bb
        \rangle, \bx_1 \in \langle \ba, \bb \rangle^{\perp}, \\
        \by &= \by_0 + \by_1, \qquad \text{with } \by_0 \in \langle \ba, \bb
        \rangle, \by_1 \in \langle \ba, \bb \rangle^{\perp}
    .\end{align*}

    If \(\bx_0 = \by_0 = 0\), we obtain a contradiction with \(\sigma(\bV) =
    \overline{\ba}\, \overline{\bb}\).
    Assume instead that \(\bx_0 \neq 0\). Then in \(\B_3'\) we have
    \[
    \overline{\ba}\, \overline{\bb} = \overline{\bx}\, \overline{\by} 
    = \overline{\bx}\, \overline{\ba}\, \overline{\bb} 
    = (\overline{\bx}_0 + \overline{\bx}_1)\, \overline{\ba}\, \overline{\bb} 
    = \overline{\ba}\, \overline{\bb} + \overline{\bx}_1\, \overline{\ba}\,
    \overline{\bb}
    .\]
    Hence \(\overline{\bx}_1\, \overline{\ba}\, \overline{\bb} = 0\) in \(\B_3'\),
    which implies \(\bx_1 = 0\). Therefore, \(\bx \in \langle \ba, \bb \rangle\).

    Similarly, if \(\by_0 \neq 0\), then \(\by \in \langle \ba, \bb \rangle\)
    and we are done. Otherwise, if \(\by_0 = 0\), we obtain a contradiction with
    \(\bx \cdot \by = 1\). The lemma follows.
\end{proof}

\begin{corollary}\label{cor:mod-2-decomposition}
    For \(g \geq 3\), let \(V \subset H_1(S_g)\) be a symplectic
    summand of rank \(2\) with \(\sigma(V) = \overline{\ba}\, \overline{\bb}\)
    for some \(a, b \in H_1(S_g)\) satisfying \(a \cdot b = 1\). Then
    there exists a symplectic basis \(x, y\) for \(V\) such that \(\bx = \ba\)
    and \(\by = \bb\).
\end{corollary}

\begin{lemma} \label{lem:gen-3-subsurface}
    Let \(g \geq 4\), and let \(W \subset H_1(S_g)\) be a symplectic
    summand of rank \(6\). Suppose \(X_1, X_2, Y_1, Y_2 \subset W\) are
    symplectic summands of rank \(2\) such that
    \begin{itemize}
        \item \(X_1 \perp X_2, Y_1 \perp Y_2\) and \(Y_2 \subset X_1 \oplus X_2\);
        \item \(\sigma(X_1) = \sigma(Y_1)\) and \(\sigma(X_2) = \sigma(Y_2)\).
    \end{itemize}
    Then in \(H_2(\I_g)\) we have
    \[
    \A(X_1, X_2) = \A(Y_1, Y_2)
    .\]
\end{lemma}

\begin{proof}
    Let \(X_3, Y_3 \subset W\) be symplectic summands of rank \(2\) such that
    \begin{equation} \label{eq:decomp}
        \begin{aligned}
            W &= X_1 \oplus X_2 \oplus X_3, \\
            W &= Y_1 \oplus Y_2 \oplus Y_3
        .\end{aligned}
    \end{equation}
    Since \(\sigma(W^{\perp}) = \sigma(W)\),
    Proposition~\ref{prop:key-relation-lin} implies that
    \begin{align*}
        \A(X_1, X_2) + \A(X_3, X_2) &= \A(W^{\perp}, X_2), \\
        \A(Y_1, Y_2) + \A(Y_3, Y_2) &= \A(W^{\perp}, Y_2) 
    .\end{align*}
    Since \(\sigma(X_2) = \sigma(Y_2)\),
    Proposition~\ref{prop:key-relation-lin} implies \(\A(W^{\perp}, X_2) =
    \A(W^{\perp}, Y_2)\). Therefore, the desired equality is equivalent to
    \[
    \A(X_3, X_2) = \A(Y_3, Y_2)
    .\]
    Since \(Y_2 \subset X_1 \oplus X_2\), we have \(Y_2 \perp X_3\), so
    \(\A(X_3, Y_2)\) is well-defined. Proposition~\ref{prop:key-relation-lin}
    then gives \(\A(X_3, X_2) = \A(X_3, Y_2)\), as \(\sigma(X_2) =
    \sigma(Y_2)\). Thus it remains to prove that
    \[
    \A(X_3, Y_2) = \A(Y_3, Y_2)
    .\]
    From \eqref{eq:decomp} we have \(\sigma(X_3) = \sigma(Y_3)\), so the
    equality follows from Proposition~\ref{prop:key-relation-lin}.
\end{proof}

Now we are ready to prove Proposition~\ref{prop:injective-2}.

\begin{proof}[Proof of Proposition~\ref{prop:injective-2}]
    Let \(U_1, U_2, V_1, V_2 \subset H_1(S_g)\) be symplectic summands of
    rank \(2\), bounded by the curves \(\delta_1, \delta_2, \gamma_1,
    \gamma_2\), respectively. Let \(a_1, b_1\) be a symplectic basis for
    \(U_1\), and \(a_2, b_2\) a symplectic basis for \(U_2\). Then we have
    \begin{align*}
        \sigma(V_1) &= \sigma(U_1) = \overline{\ba}_1 \overline{\bb}_1, \\
        \sigma(V_2) &= \sigma(U_2) = \overline{\ba}_2 \overline{\bb}_2
    .\end{align*}

    By Corollary~\ref{cor:mod-2-decomposition} applied to \(V_2\), there
    exists a symplectic basis \(x_2, y_2\) for \(V_2\) such that, for some
    \(v, w \in H_1(S_g)\), we have
    \begin{align*}
        x_2 &= a_2 + 2 v, \\
        y_2 &= b_2 + 2 w
    .\end{align*}
    Thus, we can choose \(a_3, b_3 \in \langle a_1, b_1, a_2, b_2
    \rangle^\perp\) such that \(a_3 \cdot b_3 = 1\) and
    \[
    x_2 = 2 \zeta_1 a_1 + 2 \eta_1 b_1 + (2 \zeta_2 + 1) a_2 + 2 \eta_2 b_2 + 2 \zeta_3 a_3
    .\]
    Let
    \[
    a_2' = (2\alpha_1 + 1)a_2 + 2\alpha_2 b_2 + 2\alpha_3 a_3,
    \quad \text{with } \gcd(2\alpha_1 + 1, \alpha_2, \alpha_3) = 1
    ,\]
    be a primitive representative of \((2\zeta_2 + 1)a_2 + 2\eta_2 b_2 + 2\zeta_3 a_3\).
    Then \(\gcd(2\alpha_1 + 1, 4\alpha_2, 4\alpha_3) = 1\), and there exist integers
    \(\beta_1, \beta_2, \beta_3 \in \Z\) such that
    \[
    (2\alpha_1 + 1)(2\beta_1 + 1) + 4\alpha_2 \beta_2 + 4\alpha_3 \beta_3 = 1
    .\]
    Let
    \[
    b_2' = (2 \beta_1 + 1) b_2 - 2 \beta_2 a_2 + 2 \beta_3 b_3.
    \]
    Then \(a_2' \cdot b_2' = 1\), and \(U_2' = \langle a_2', b_2' \rangle\)
    satisfies \(U_2' \perp U_1\). From Proposition~\ref{prop:key-relation-lin},
    it follows that \(\A(U_1, U_2) = \A(U_1, U_2')\).

    For some integer \(\zeta_2' \in \Z\), we have
    \[
    x_2 = 2 \zeta_1 a_1 + 2 \eta_1 b_1 + (2 \zeta_2' + 1) a_2'
    ,\]
    and we can choose \(a_3' \in \langle a_1, b_1, a_2', b_2'
    \rangle^{\perp}\) such that 
    \[
    y_2 = 2 \lambda_1 a_1 + 2 \mu_1 b_1 + 2 \lambda_2' a_2' + (2 \mu_2' + 1)
    b_2' + 2 \lambda_3' a_3'
    .\]
    
    There exist \(x_1', y_1' \in H_1(S_g)\) such that:
    \begin{itemize}
        \item \(\bx_1' = \ba_1\) and \(\by_1' = \bb_1\);
        \item \(\langle x_1', y_1' \rangle \perp V_2\);
        \item \(x_1' \cdot y_1' = 1\).
    \end{itemize}
    Indeed, we may set
    \begin{align*}
        x_1' &= (2 \zeta_2' + 1) (2 \mu_2' + 1) a_1 + (2 \mu_2' + 1) 2 \eta_1
        b_2' + \left( 2 \eta_1 \cdot 2 \lambda_2' - (2 \zeta_2' + 1) \cdot 2
        \mu_1 \right) a_2'
        + a_4'\\
        y_1' &= (2 \zeta_2' + 1) (2 \mu_2' + 1) b_1 - (2 \mu_2' + 1) 2 \zeta_1 b_2' +
        \left(-2 \zeta_1 \cdot 2 \lambda_2' + (2 \zeta_2' + 1) \cdot 2 \lambda_1  \right) a_2' + 2 \nu
        b_4'
    ,\end{align*}
    where \(a_4', b_4' \in \langle a_1, b_1, a_2', b_2', a_3', b_3'
    \rangle^{\perp}\) with \(a_4' \cdot b_4' = 1\), and the integer constant
    \(\nu\) is chosen so that \(x_1' \cdot y_1' = 1\). This is possible
    because 
    \[
    (2 \zeta_2' + 1)^2 (2 \mu_2' + 1)^2 \equiv 1 \pmod{4}
    .\] 
    It follows that, for \(V_1' = \langle x_1', y_1' \rangle\), we have
    \[
    \sigma(V_1') = \sigma(V_1) = \overline{\ba}_1 \overline{\bb}_1
    .\]
    Proposition~\ref{prop:key-relation-lin} then implies that
    \(\A(V_1', V_2) = \A(V_1, V_2)\).
    Let \(y_2' = y_2 - 2 \lambda_3' a_3'\) and set \(V_2' = \langle x_2,
    y_2' \rangle\). Then, again by Proposition~\ref{prop:key-relation-lin}, we
    obtain \(\A(V_1', V_2) = \A(V_1', V_2')\).

    The proposition now follows from Lemma~\ref{lem:gen-3-subsurface} applied
    to \(U_1, U_2', V_1', V_2' \subset W\), where \(W = \langle a_1, b_1, a_2', b_2', a_4',
    b_4' \rangle\).
\end{proof}

%% file: sections/order-2-in-hk.tex
\section{Abelian cycles in \(H_k(\I_g)\)} \label{sec:order-2-hk}

In this section, we extend several results on
\(H_2^{\mathrm{ab,sep}}(\I_g)\) to the higher homology groups \(H_k(\I_g)\).

\begin{proposition} \label{prop:abelian-cycle-order-higher}
    Let \(\delta_1, \dots, \delta_k\) be pairwise disjoint, pairwise
    nonisotopic separating simple closed curves on \(S_g\). Then, for \(g
    \geq 3\) and \(3 \leq k \leq 2g-3\), we have
    \[
    2 \A(T_{\delta_1}, \dots, T_{\delta_k}) = 0
    .\] 
\end{proposition}

\begin{remark}
    Combining
    Propositions~\ref{prop:johnson3},~\ref{prop:abelian-cycle-order},~and~\ref{prop:abelian-cycle-order-higher},
    we obtain \(2 \A(T_{\delta_1}, \dots, T_{\delta_k}) = 0\) for \(1
    \leq k \leq 2g-3\).
\end{remark}

We introduce some terminology needed for the proof of
Proposition~\ref{prop:abelian-cycle-order-higher}. Throughout this section, we
assume that the surface \(S_g\) has genus at least \(3\).

We call a set of curves \(\{\delta_1, \dots, \delta_k\}\) on \(S_g\) an
\emph{admissible partition} if the curves \(\delta_1, \dots, \delta_k\)
are pairwise disjoint, pairwise nonisotopic, separating simple closed curves.
For convenience, we will also say that \(\delta_1, \dots, \delta_k\) form
an admissible partition, implicitly referring to the set
\(\{\delta_1, \dots, \delta_k\}\).

Let \(\delta_1, \dots, \delta_k\) be an admissible partition of \(S_g\), and
fix this set of curves. Denote by \(\SSS\) the set of surfaces obtained by
cutting \(S_g\) along \(\delta_1 \cup \cdots \cup \delta_k\).

We call a curve \(\delta_j\) \emph{outermost} if there exists a surface
\(\Sigma \in \SSS\) such that \(\partial \Sigma = \delta_j\).
We call \(\Sigma\) a \emph{cap over \(\delta_j\)} and denote it by \(\Cap(\delta_j)\).
The set of outermost curves among \(\delta_1, \dots, \delta_k\) is denoted
by \(\OP(\{\delta_1, \dots, \delta_k\})\), or simply by \(\OP\) when the context
makes clear which admissible partition is considered.

An admissible partition \(\delta_1, \dots, \delta_k\) of \(S_g\) is called
\emph{simple} if, for each \(\delta_i \in \OP\), we have \(g(\Cap(\delta_i)) = 1\).

We call a curve \(\delta_i\) \emph{grouping} if \(\delta_i\) is not outermost
and there exists a surface \(\Sigma \in \SSS\) such that \(\delta_i \subset
\partial \Sigma\) and \(\partial \Sigma
\subset \delta_i \cup \OP\); we call \(\Sigma\) a \emph{cap base of
\(\delta_i\)} and denote it by \(\CB(\delta_i)\). In other words, \(\delta_i\)
is grouping if it bounds a subsurface containing only outermost curves. The
cap base is uniquely defined when there are at least two grouping curves; if
there is exactly one grouping curve, then either of the two possible surfaces
may be chosen.

For a grouping curve \(\delta_i\), let \(\OP_i\) denote the subset of
\(\OP\) consisting of curves lying in \(\CB(\delta_i)\).
We then say that \(\delta_i\) \emph{groups the curves} \(\delta_{i_1}, \dots, \delta_{i_k}\),
where \(\OP_i = \{\delta_{i_1}, \dots, \delta_{i_k}\}\).
For such a curve \(\delta_i\), we define
\[
\GCap(\delta_i) = \CB(\delta_i) \cup \Cap(\delta_{i_1}) \cup \cdots \cup \Cap(\delta_{i_k}).
\]
The set of all grouping curves among \(\delta_1, \dots, \delta_k\) is denoted by
\(\GP(\{\delta_1, \dots, \delta_k\})\), or simply by \(\GP\) when the context
is clear.

\begin{lemma} \label{lem:gen-0-part-curve}
    Let \(\delta_1, \dots, \delta_k\) be a simple admissible partition
    of \(S_g\). 
    If there exists a grouping curve \(\delta_i \in \GP\) that has a cap base
    of genus~\(0\), then
    \(\A(T_{\delta_1}, \dots, T_{\delta_k}) = 0\) in \(H_k(\I_g)\).
\end{lemma}

\begin{proof}
    If there are at least two grouping curves, then the cap base of
    \(\delta_i\) is uniquely defined. If \(\delta_i\) is the only grouping
    curve, we denote by \(\CB(\delta_i)\) a cap base of genus~\(0\) (if both
    cap bases have genus~\(0\), either may be chosen).

    Without loss of generality, assume that the grouping curve \(\delta_1\),
    which groups the curves \(\delta_2, \dots, \delta_m\) (with \(m
    \geq 3\)), satisfies \(g(\CB(\delta_1)) = 0\). The proof proceeds by
    induction on \(m\).

    \textit{Base case.} For \(m = 3\), the statement follows from
    Lemma~\ref{lem:sum-relation-1} applied to \(\delta_1, \delta_2,
    \delta_3\). Indeed, we have
    \[
    \A(T_{\delta_1}, \dots, T_{\delta_k}) + \A(T_{\delta_3}, T_{\delta_2},
    T_{\delta_3}, \dots, T_{\delta_k})
    = \A(T_{\delta_2}, T_{\delta_2}, T_{\delta_3}, \dots, T_{\delta_k})
    ,\]
    which gives \(\A(T_{\delta_1}, \dots, T_{\delta_k}) = 0\).

    \textit{Induction step.} Assume that the statement holds for all \(m
    \leq \ell-1\), and consider the case \(m = \ell\). Choose a separating
    simple closed curve \(\gamma \subset \CB(\delta_1)\) such that \(\delta_1
    \cup \delta_2 \cup \gamma\) bounds a pair of pants. Then, by
    Lemma~\ref{lem:sum-relation-1}, we have
    \[
    \A(T_{\delta_1}, \dots, T_{\delta_k}) + \A(T_{\gamma}, T_{\delta_2}, \dots, T_{\delta_k})
    = \A(T_{\delta_2}, T_{\delta_2}, \dots, T_{\delta_k})
    ,\]
    thus \(\A(T_{\delta_1}, \dots, T_{\delta_k}) = -\A(T_{\gamma},
    T_{\delta_2}, \dots, T_{\delta_k})\). Since \(\gamma\) groups \(\delta_3,
    \dots, \delta_k\), the induction hypothesis implies \(\A(T_{\gamma},
    T_{\delta_2}, \dots, T_{\delta_k}) = 0\), yielding the desired result.
\end{proof}

\begin{lemma} \label{lem:change-cb-gen-1}
    Let \(\delta_1, \dots, \delta_k\) be a simple admissible partition
    of \(S_g\) such that all surfaces obtained by cutting \(S_g\) along
    \(\delta_1, \dots, \delta_k\) have genus~\(1\).
    Then there exists a simple admissible partition
    \(\delta_1', \dots, \delta_k'\) such that:
    \begin{enumerate}
        \item none of the curves \(\delta_1', \dots, \delta_k'\) is a grouping curve;
        \item \(\A(T_{\delta_1}, \dots, T_{\delta_k}) =
              \A(T_{\delta_1'}, \dots, T_{\delta_k'})\).
    \end{enumerate}
\end{lemma}

\begin{proof}
    It suffices to show the following:
    for a curve \(\delta_1\) that groups the curves \(\delta_2, \dots, \delta_m\),
    there exists an outermost separating simple closed curve \(\delta_1'\) on
    \(\CB(\delta_1)\) that is disjoint from \(\delta_1, \dots, \delta_k\) and such that
    \[
    \A(T_{\delta_1}, T_{\delta_2}, \dots, T_{\delta_k}) = \A(T_{\delta_1'},
    T_{\delta_2}, \dots, T_{\delta_k})
    .\]
    Since \(g(\CB(\delta_1)) = 1\), the hypotheses of the lemma are satisfied
    for the curves \(\delta_1', \delta_2, \dots, \delta_k\).
    
    Choose an outermost separating simple closed curve \(\delta_1'\) of genus
    \(1\) and a separating simple closed curve \(\gamma\) on
    \(\CB(\delta_1)\), as illustrated in
    Figure~\ref{fig:gen1-grouping-curve}.

    \begin{figure}[H]
        \centering
    \def\svgwidth{.5\columnwidth}
    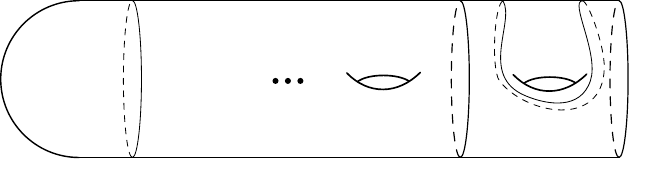

        \caption{The curves \(\delta_1'\) and \(\gamma\) on \(\GCap(\delta_1)\).}
        \label{fig:gen1-grouping-curve}
    \end{figure}

    Apply Lemma~\ref{lem:sum-relation-1} to the curves \(\gamma, \delta_1', \delta_1\):
    \[
    \A(T_{\delta_1}, T_{\delta_2}, \dots, T_{\delta_k}) + \A(T_{\gamma}, T_{\delta_2}, \dots, T_{\delta_k}) 
    = \A(T_{\delta_1'}, T_{\delta_2}, \dots, T_{\delta_k})
    .\]
    By Lemma~\ref{lem:gen-0-part-curve}, we have \(\A(T_{\gamma},
    T_{\delta_2}, \dots, T_{\delta_k}) = 0\) in \(H_k(\I_g)\). Hence,
    \[
    \A(T_{\delta_1}, \dots, T_{\delta_k}) = \A(T_{\delta_1'}, T_{\delta_2}, \dots, T_{\delta_k})
    ,\]
    and the lemma follows.
\end{proof}

\begin{lemma} \label{lem:no-partitioning-curves}
    Let \(\delta_1, \dots, \delta_k\) be a simple admissible partition of \(S_g\),
    and assume that none of the curves \(\delta_1, \dots, \delta_k\) is a
    grouping curve. Then in \(H_k(\I_g)\) we have
    \[
    2 \A(T_{\delta_1}, \dots, T_{\delta_k}) = 0
    .\]
\end{lemma}

\begin{proof}
    Since there are no grouping curves, all curves \(\delta_1, \dots,
    \delta_k\) are outermost. Choose a separating simple closed curve
    \(\gamma\) on \(S_g\), disjoint from \(\delta_1, \dots, \delta_k\), such
    that \(\delta_1 \cup \delta_2 \cup \gamma\) bounds a pair of pants. Then,
    applying Lemma~\ref{lem:sum-relation-1} to the curves \(\gamma, \delta_1,
    \delta_2\), we obtain
    \[
    \A(T_{\delta_1}, T_{\delta_2}, \dots, T_{\delta_k}) +
    \A(T_{\gamma}, T_{\delta_2}, \dots, T_{\delta_k}) =
    \A(T_{\delta_2}, T_{\delta_2}, \dots, T_{\delta_k}).
    \]
    That is, \(\A(T_{\delta_1}, \dots, T_{\delta_k}) = \A(T_{\gamma},
    T_{\delta_2}, \dots, T_{\delta_k})\). The curve \(\gamma\) bounds a
    genus-2 subsurface that contains \(\delta_2\). Then, by point~(2) of
    Lemma~\ref{lem:gen-2-relation}, we have \(2 \A(T_{\gamma}, T_{\delta_2},
    \dots, T_{\delta_k}) = 0\) in \(H_k(\I_g)\), which completes the proof.
\end{proof}

\begin{proof}[Proof of Proposition~\ref{prop:abelian-cycle-order-higher}]
    It follows from Lemma~\ref{lem:gen-2-relation} that
    \(2 \A(T_{\delta_1}, \dots, T_{\delta_k}) = 0\) in \(H_k(\I_g)\)
    if any of the following conditions hold:
    \begin{itemize}
        \item for some outermost curve \(\delta_i \in \OP\), we have
        \(g(\Cap(\delta_i)) \geq 2\);
        \item for some grouping curve \(\delta_j \in \GP\), we have
        \(g(\CB(\delta_j)) \geq 2\).
    \end{itemize}
    
    We now assume that \(g(\Cap(\delta_i)) = 1\) for all \(\delta_i \in \OP\) and
    \(g(\CB(\delta_j)) \leq 1\) for all \(\delta_j \in \GP\). If
    \(g(\CB(\delta)) = 0\) for some grouping curve \(\delta\), then the claim
    follows from Lemma \ref{lem:gen-0-part-curve}. Otherwise, if
    \(g(\CB(\delta)) = 1\) for all \(\delta \in \GP\), we first apply Lemma
    \ref{lem:change-cb-gen-1} to obtain an admissible partition \(\delta_1',
    \dots, \delta_k'\), and then apply Lemma \ref{lem:no-partitioning-curves}
    to this partition.
\end{proof}

\begin{corollary}
    For \(g \geq 3\) and \(3 \leq k \leq 2g-3\), the group
    \(H_k^{\mathrm{ab,sep}}(\I_g)\) is a \(\Z / 2\Z\)-vector space.
\end{corollary}

As another consequence, we obtain the analogue of Proposition~\ref{prop:ab-sep-gen}.

\begin{proposition} \label{prop:ab-sep-gen-k}
    For \(g \geq 3\), the following hold:
    \begin{itemize}
        \item if \(k < g\), then the \(\Z/2\Z\)-vector space \(H_k^{\mathrm{ab,sep}}(\I_g)\) is
        generated by \(\A(T_{\delta_1}, \dots, T_{\delta_k})\),
        where \(\delta_1, \dots, \delta_k\) are pairwise disjoint, pairwise
        nonisotopic, separating simple closed curves of genus \(1\);
        
        \item if \(k \geq g\), then the \(\Z / 2\Z\)-vector space
        \(H_k^{\mathrm{ab,sep}}(\I_g)\) is trivial.
    \end{itemize}
\end{proposition}

\begin{proof}
    First, we show that an arbitrary abelian cycle \(\A(T_{\delta_1}, \dots,
    T_{\delta_k})\) can be expressed as a sum of abelian cycles
    \(\A(T_{\theta_{1,i}}, \dots, T_{\theta_{k,i}})\), each satisfying 
    \[
    \max \{g(\theta_{j,i}) \mid \theta_{j,i} \in \OP(\theta_{1,i}, \dots,
    \theta_{k,i})\} = 1
    .\] 
    It suffices to show that, for the outermost curve (without loss of
    generality, \(\delta_1\)), there exist separating simple closed curves
    \(\varepsilon_1, \dots, \varepsilon_n\) of genus \(1\), disjoint from
    \(\delta_2, \dots, \delta_k\), such that
    \[
    \A(T_{\delta_1}, \dots, T_{\delta_k}) = \sum_{i=1}^{n} \A(T_{\varepsilon_i}, T_{\delta_2},
    \dots, T_{\delta_k})
    .\]
    The same argument as in the proof of Proposition~\ref{prop:ab-sep-gen}
    applies here.

    Thus, it remains to prove the statement for the abelian cycle
    \(\A(T_{\delta_1}, \dots, T_{\delta_k})\), where 
    \[
    \max \{ g(\delta) \mid \delta \in \OP \} = 1
    .\] 
    If none of the curves \(\delta_1, \dots,
    \delta_k\) is grouping, the claim follows immediately.  Otherwise,
    consider a grouping curve, which we may assume without loss of generality
    to be \(\delta_1\). If \(g(\CB(\delta_1)) = 0\), then \(\A(T_{\delta_1},
    \dots, T_{\delta_k}) = 0\) by Lemma~\ref{lem:gen-0-part-curve}. If
    \(g(\CB(\delta_1)) > 0\), we can choose separating simple closed curves
    \(\gamma_1, \dots, \gamma_n \subset \CB(\delta_1)\) of genus \(1\) such
    that in \(H_k(\I_g)\) we have
    \[
    \A(T_{\delta_1}, \dots, T_{\delta_k}) = 
    \sum_{i=1}^{n} \A(T_{\gamma_i}, T_{\delta_2}, \dots, T_{\delta_k})
    .\]
    Indeed, we proceed in the same way as in the proof of
    Proposition~\ref{prop:ab-sep-gen}. Choose separating simple closed curves
    \(\gamma_1, \dots, \gamma_{n+1}\) and \(\eta_1, \dots, \eta_{n-1}\) on
    \(\CB(\delta_1)\), as shown in Figure~\ref{fig:gen-1-generation-higher}.

    \begin{figure}[H]
        \centering
    \def\svgwidth{.6\columnwidth}
    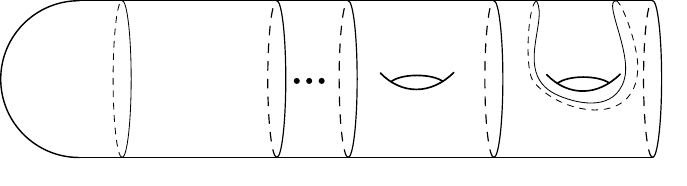

        \caption{The curves on \(\CB(\delta_1)\).}
        \label{fig:gen-1-generation-higher}
    \end{figure}
    Then, by Lemma~\ref{lem:sum-relation-1}, we have
    \[
    \A(T_{\delta_1}, T_{\delta_2}, \dots, T_{\delta_k}) = \A(T_{\gamma_1}, T_{\delta_2}, \dots,
    T_{\delta_k}) + \A(T_{\eta_1}, T_{\delta_2}, \dots, T_{\delta_k})
    .\]

    Applying the same procedure repeatedly to \(\A(T_{\eta_1}, T_{\delta_2},
    \dots, T_{\delta_k})\) and subsequent terms, we obtain the decomposition 
    \[
    \A(T_{\delta_1}, \dots, T_{\delta_k}) =
    \sum_{i=1}^{n+1} \A(T_{\gamma_i}, T_{\delta_2}, \dots, T_{\delta_k}) .
    \]
    Since \(\A(T_{\gamma_{n+1}}, T_{\delta_2}, \dots, T_{\delta_k}) = 0\) by
    Lemma~\ref{lem:gen-0-part-curve}, it follows that
    \[
    \A(T_{\delta_1}, \dots, T_{\delta_k}) =
    \sum_{i=1}^{n} \A(T_{\gamma_i}, T_{\delta_2}, \dots, T_{\delta_k}) .
    \]
    This yields the desired result.

    Repeating the procedure described above for each grouping curve \(\delta_i
    \in \GP\), we obtain that the abelian cycle \(\A(T_{\delta_1}, \dots,
    T_{\delta_k})\) can be expressed as a sum of abelian cycles
    \(\A(T_{\gamma_{j_1}}, \dots, T_{\gamma_{j_k}})\), where all outermost
    curves among \(\gamma_{j_1}, \dots, \gamma_{j_k}\) have genus~\(1\).

    Hence, if \(k < g\), the claim follows. If \(k \geq g\), there is
    necessarily at least one grouping curve among \(\gamma_{j_1}, \dots,
    \gamma_{j_k}\), and the resulting sum is empty, i.e., zero, as desired.
\end{proof}

Theorem~\ref{thm:main-1} now follows from
Propositions~\ref{prop:abelian-cycle-order},~\ref{prop:ab-cyc-ord-gen-3},~\ref{prop:ab-sep-gen},~\ref{prop:abelian-cycle-order-higher},
and \ref{prop:ab-sep-gen-k}.

\bigskip
\noindent
We conclude the paper by proving a criterion for when an abelian cycle
\(\A(T_{\delta_1}, \dots, T_{\delta_k})\) vanishes in \(H_k(\I_g)\) for \(k <
g\). 

For the proof of Theorem~\ref{thm:main-3}, we use the homomorphism 
\(\sigma_k: H_k^{\mathrm{ab,sep}}(\I_g) \to \wedge^k \B_2'\), which is
induced by the Birman–Craggs–Johnson homomorphism.  
As in the case \(k = 2\), we have
\[
\sigma_k(\A(T_{\delta_1}, \dots, T_{\delta_k})) = \sigma(T_{\delta_1}) \wedge
\cdots \wedge \sigma(T_{\delta_k})
.\]

\begin{proof}[Proof of Theorem~\ref{thm:main-3}]
    First, we prove that \(\A(T_{\delta_1}, \dots, T_{\delta_k}) = 0\) if at
    least one of the subsurfaces obtained by cutting \(S_g\) along \(\delta_1,
    \dots, \delta_k\) has genus \(0\).

    Denote by \(\SSS\) the set of subsurfaces obtained by cutting \(S_g\)
    along \(\delta_1, \dots, \delta_k\). Consider a graph \(\Sfr\) whose
    vertices are the elements of \(\SSS\), with an edge between two
    subsurfaces if they share a boundary component. Note that \(\Sfr\)
    is in fact a tree, since \(\delta_1, \dots, \delta_k\) are all separating
    curves. By the genus of a vertex \(\Sigma \in \Sfr\) we mean a genus of
    the subsurface \(\Sigma\).

    We will call a vertex \(\Sigma \in \Sfr\) of genus \(0\)
    \begin{itemize}
        \item \emph{good} if at least one of its adjacent vertices has genus at least \(1\);
        \item \emph{bad} otherwise, i.e.\ if all its neighbors have genus \(0\).
    \end{itemize}

    By the hypothesis of the theorem, there exists a vertex \(\Sigma \in
    \Sfr\) of genus \(0\). We claim that a good vertex must exist.
    Indeed, suppose that no good vertices exist. Then, starting from
    \(\Sigma\) and traversing the tree, every vertex would have genus \(0\).
    This would contradict the fact that the original surface \(S_g\) has
    positive genus.
    
    Let \(\Sigma\) be a good vertex of genus \(0\). Relabel the curves
    \(\delta_1, \dots, \delta_k\) so that:
    \begin{itemize}
        \item \(\delta_1, \dots, \delta_m\) (with \(m \leq k\)) bound
        \(\Sigma\); and
        \item  the subsurface \(\Sigma_1 \in \SSS \setminus \{\Sigma\}\) with
        \(\delta_1 \subset \partial \Sigma_1\) has genus at least \(1\).
    \end{itemize}
    The proof proceeds by induction on \(m\).
    
    \textit{Base case.}
    Let \(m = 3\). Then we can apply Lemma~\ref{lem:sum-relation-2} to the
    curves \(\delta_1, \delta_2, \delta_3\), obtaining
    \[
    \A(T_{\delta_2}, T_{\delta_2}, \dots, T_{\delta_k})
    + \A(T_{\delta_3}, T_{\delta_2}, T_{\delta_3}, \dots, T_{\delta_k})
    = \A(T_{\delta_1}, \dots, T_{\delta_k})
    ,\]
    from which it follows that \(\A(T_{\delta_1}, \dots, T_{\delta_k}) = 0\).

    \textit{Inductive step.}
    Suppose the statement holds for all \(m \leq \ell\), and let us show
    that it holds for \(m = \ell + 1\). Choose a curve \(\gamma\) such that
    \(\delta_1, \delta_2, \gamma\) bound a genus-\(0\) subsurface. Then, by
    Lemma~\ref{lem:sum-relation-2}, we have
    \[
    \A(T_{\gamma}, T_{\delta_2}, \dots, T_{\delta_k})
    + \A(T_{\delta_2}, T_{\delta_2}, \dots, T_{\delta_k})
    = \A(T_{\delta_1}, \dots, T_{\delta_k})
    ,\]
    that is 
    \[
    \A(T_{\delta_1}, \dots, T_{\delta_k})
    = \A(T_{\gamma}, T_{\delta_2}, \dots, T_{\delta_k})
    .\]
    By the induction hypothesis applied to the \(m-1\) curves \(\gamma,
    \delta_3, \dots, \delta_m\), we conclude that \(\A(T_{\gamma},
    T_{\delta_2}, \dots, T_{\delta_k}) = 0\), and hence
    \(\A(T_{\delta_1}, \dots, T_{\delta_k}) = 0\).

    \bigskip
    To prove the statement in the reverse direction, it suffices to show that
    \[
    \sigma_k(\A(T_{\delta_1}, \dots, T_{\delta_k})) \neq 0
    \] 
    if all subsurfaces obtained by cutting \(S_g\) along \(\delta_1, \dots,
    \delta_k\) have positive genus.

    Using an argument similar to the proof of Lemma~\ref{lem:change-cb-gen-1},
    one can show that there exists an admissible partition \(\delta_1',
    \dots, \delta_k'\) such that:
    \begin{itemize}
        \item there are no grouping curves among \(\delta_1', \dots, \delta_k'\);
        \item all subsurfaces obtained by cutting \(S_g\) along the curves
        \(\delta_1', \dots, \delta_k'\) have positive genus;
        \item \(\A(T_{\delta_1}, \dots, T_{\delta_k}) = \A(T_{\delta_1'}, \dots, T_{\delta_k'})\).
    \end{itemize}
    Therefore, we may assume that all curves \(\delta_1, \dots, \delta_k\) are
    outermost and that the surface \(S_g \setminus \left(
    \Cap(\delta_1) \cup \dots \cup \Cap(\delta_k) \right)\) has positive
    genus. Under this assumption, we can choose a symplectic basis \(\ba_1,
    \dots, \ba_g, \bb_1, \dots, \bb_g\) for \(H_1(\Sigma_g; \Z/2\Z)\) such
    that
    \begin{align*}
        \sigma(T_{\delta_1}) &= \overline{\ba}_1 \overline{\bb}_1 + \cdots +
        \overline{\ba}_{i_1} \overline{\bb}_{i_1}, \\
        \sigma(T_{\delta_2}) &= \overline{\ba}_{i_1+1} \overline{\bb}_{i_1+1}
        + \cdots + \overline{\ba}_{i_2} \overline{\bb}_{i_2}, \\
        &\vdots \\
        \sigma(T_{\delta_k}) &= \overline{\ba}_{i_{k-1}+1}
        \overline{\bb}_{i_{k-1}+1}
        + \cdots + \overline{\ba}_{i_k} \overline{\bb}_{i_k}
    ,\end{align*}
    where \(1 < i_1 < \dots < i_k < g\). Then
    \begin{align*}
        \sigma_k(\A(T_{\delta_1}, \dots, T_{\delta_k})) &= 
        \sigma(T_{\delta_1}) \wedge \cdots \wedge \sigma(T_{\delta_k}) \\
        &=
        (\overline{\ba}_1 \overline{\bb}_1 + \cdots +
        \overline{\ba}_{i_1} \overline{\bb}_{i_1}) \wedge \cdots \wedge
        (\overline{\ba}_{i_{k-1}+1}
        \overline{\bb}_{i_{k-1}+1}
        + \cdots + \overline{\ba}_{i_k} \overline{\bb}_{i_k})
    .\end{align*}
    Since \(i_k < g\), the elements
    \[
    \overline{\ba}_1 \overline{\bb}_1, \dots, \overline{\ba}_{i_k}
    \overline{\bb}_{i_k}
    \] 
    are linearly independent in \(\B_2'\). It then follows that expanding the
    parentheses yields a sum of linearly independent vectors, so
    \[
    \sigma_k(\A(T_{\delta_1}, \dots, T_{\delta_k})) \neq 0
    ,\]
    which implies \(\A(T_{\delta_1}, \dots, T_{\delta_k}) \neq 0\).
\end{proof}

%% file: figures/gen1-grouping-curve.pdf_tex
\begingroup%
  \makeatletter%
  \providecommand\color[2][]{%
    \errmessage{(Inkscape) Color is used for the text in Inkscape, but the package 'color.sty' is not loaded}%
    \renewcommand\color[2][]{}%
  }%
  \providecommand\transparent[1]{%
    \errmessage{(Inkscape) Transparency is used (non-zero) for the text in Inkscape, but the package 'transparent.sty' is not loaded}%
    \renewcommand\transparent[1]{}%
  }%
  \providecommand\rotatebox[2]{#2}%
  \newcommand*\fsize{\dimexpr\f@size pt\relax}%
  \newcommand*\lineheight[1]{\fontsize{\fsize}{#1\fsize}\selectfont}%
  \ifx\svgwidth\undefined%
    \setlength{\unitlength}{312.96837274bp}%
    \ifx\svgscale\undefined%
      \relax%
    \else%
      \setlength{\unitlength}{\unitlength * \real{\svgscale}}%
    \fi%
  \else%
    \setlength{\unitlength}{\svgwidth}%
  \fi%
  \global\let\svgwidth\undefined%
  \global\let\svgscale\undefined%
  \makeatother%
  \begin{picture}(1,0.29145123)%
    \lineheight{1}%
    \setlength\tabcolsep{0pt}%
    \put(0,0){\includegraphics[width=\unitlength,page=1]{gen1-grouping-curve.pdf}}%
    \put(0.95801496,0.00221427){\color[rgb]{0,0,0}\makebox(0,0)[t]{\lineheight{1.25}\smash{\begin{tabular}[t]{c}$\delta_1$\end{tabular}}}}%
    \put(0,0){\includegraphics[width=\unitlength,page=2]{gen1-grouping-curve.pdf}}%
    \put(0.55227323,0.2376798){\color[rgb]{0,0,0}\makebox(0,0)[t]{\lineheight{1.25}\smash{\begin{tabular}[t]{c}$\delta_m$\end{tabular}}}}%
    \put(0,0){\includegraphics[width=\unitlength,page=3]{gen1-grouping-curve.pdf}}%
    \put(0.30189157,0.23826606){\color[rgb]{0,0,0}\makebox(0,0)[t]{\lineheight{1.25}\smash{\begin{tabular}[t]{c}$\delta_3$\end{tabular}}}}%
    \put(0.81132581,0.24352303){\color[rgb]{0,0,0}\makebox(0,0)[t]{\lineheight{1.25}\smash{\begin{tabular}[t]{c}$\delta_1'$\end{tabular}}}}%
    \put(0.71151256,0.01496325){\color[rgb]{0,0,0}\makebox(0,0)[t]{\lineheight{1.25}\smash{\begin{tabular}[t]{c}$\gamma$\end{tabular}}}}%
    \put(0.2118269,0.00531335){\color[rgb]{0,0,0}\makebox(0,0)[t]{\lineheight{1.25}\smash{\begin{tabular}[t]{c}$\delta_2$\end{tabular}}}}%
  \end{picture}%
\endgroup%

%% file: figures/gen-1-generation-higher.pdf_tex
\begingroup%
  \makeatletter%
  \providecommand\color[2][]{%
    \errmessage{(Inkscape) Color is used for the text in Inkscape, but the package 'color.sty' is not loaded}%
    \renewcommand\color[2][]{}%
  }%
  \providecommand\transparent[1]{%
    \errmessage{(Inkscape) Transparency is used (non-zero) for the text in Inkscape, but the package 'transparent.sty' is not loaded}%
    \renewcommand\transparent[1]{}%
  }%
  \providecommand\rotatebox[2]{#2}%
  \newcommand*\fsize{\dimexpr\f@size pt\relax}%
  \newcommand*\lineheight[1]{\fontsize{\fsize}{#1\fsize}\selectfont}%
  \ifx\svgwidth\undefined%
    \setlength{\unitlength}{329.33322216bp}%
    \ifx\svgscale\undefined%
      \relax%
    \else%
      \setlength{\unitlength}{\unitlength * \real{\svgscale}}%
    \fi%
  \else%
    \setlength{\unitlength}{\svgwidth}%
  \fi%
  \global\let\svgwidth\undefined%
  \global\let\svgscale\undefined%
  \makeatother%
  \begin{picture}(1,0.26966201)%
    \lineheight{1}%
    \setlength\tabcolsep{0pt}%
    \put(0,0){\includegraphics[width=\unitlength,page=1]{gen-1-generation-higher.pdf}}%
    \put(0.96010115,0.00210424){\color[rgb]{0,0,0}\makebox(0,0)[t]{\lineheight{1.25}\smash{\begin{tabular}[t]{c}$\delta_1$\end{tabular}}}}%
    \put(0,0){\includegraphics[width=\unitlength,page=2]{gen-1-generation-higher.pdf}}%
    \put(0.80629656,0.20369598){\color[rgb]{0,0,0}\makebox(0,0)[t]{\lineheight{1.25}\smash{\begin{tabular}[t]{c}$\gamma_1$\end{tabular}}}}%
    \put(0.56909793,0.20678723){\color[rgb]{0,0,0}\makebox(0,0)[t]{\lineheight{1.25}\smash{\begin{tabular}[t]{c}$\gamma_2$\end{tabular}}}}%
    \put(0.23153277,0.05617836){\color[rgb]{0,0,0}\makebox(0,0)[t]{\lineheight{1.25}\smash{\begin{tabular}[t]{c}$\gamma_{n+1}$\end{tabular}}}}%
    \put(0.04008346,0.13676158){\color[rgb]{0,0,0}\makebox(0,0)[t]{\lineheight{1.25}\smash{\begin{tabular}[t]{c}$\delta_i$\end{tabular}}}}%
    \put(0.27328227,0.18580803){\color[rgb]{0,0,0}\makebox(0,0)[t]{\lineheight{1.25}\smash{\begin{tabular}[t]{c}$\gamma_{n}$\end{tabular}}}}%
    \put(0.72524556,0.00776928){\color[rgb]{0,0,0}\makebox(0,0)[t]{\lineheight{1.25}\smash{\begin{tabular}[t]{c}$\eta_1$\end{tabular}}}}%
    \put(0.50662199,0.01232401){\color[rgb]{0,0,0}\makebox(0,0)[t]{\lineheight{1.25}\smash{\begin{tabular}[t]{c}$\eta_2$\end{tabular}}}}%
    \put(0.40083844,0.0107007){\color[rgb]{0,0,0}\makebox(0,0)[t]{\lineheight{1.25}\smash{\begin{tabular}[t]{c}$\eta_{n-1}$\end{tabular}}}}%
    \put(0,0){\includegraphics[width=\unitlength,page=3]{gen-1-generation-higher.pdf}}%
  \end{picture}%
\endgroup%